\documentclass[a4paper,11pt]{amsart}
\usepackage{amsmath,amsthm,amscd,amssymb,amsfonts, amsbsy}
\usepackage{latexsym}
\usepackage{txfonts}

\usepackage{pgf}
\usepackage{color}

\numberwithin{equation}{section}

\theoremstyle{plain}
\newtheorem{theorem}[equation]{Theorem}
\newtheorem{lemma}[equation]{Lemma}
\newtheorem{corollary}[equation]{Corollary}

\newtheorem{est}{Estimate}

\newtheorem{fact}{Fact}

\theoremstyle{definition}
\newtheorem{definition}[equation]{Definition}

\theoremstyle{remark}
\newtheorem{remark}[equation]{Remark}

\newtheorem{claim}[equation]{Claim}

\newtheorem{case}[equation]{Case}

\newcommand{\dv}{\operatorname{div}}

\newcommand{\supp}{\operatorname{supp}}
\newcommand{\dist}{\operatorname{dist}}

\newcommand{\aaviint}{{\fint\!\!\!\fint}}
\newcommand{\leftexp}[2]{ {\vphantom{#2}}^{#1}{#2}}

\newcommand{\re}{\mathbb{R}}
\newcommand{\rn}{\mathbb{R}^n}

\newcommand{\reu}{\mathbb{R}^{n+1}_+}
\newcommand{\ree}{\mathbb{R}^{n+1}}
\newcommand{\dd}{\mathbb{D}}

\newcommand{\eps}{\varepsilon}
\newcommand{\epp}{\epsilon}
\newcommand{\vp}{\varphi}
\newcommand{\vpt}{\tilde{\varphi}}
\newcommand{\pe}{\mathcal{P}_{\eta t}^*}

\newcommand{\Pe}{\mathcal{P}_{\eta t}}
\newcommand{\p}{\mathcal{P}}
\newcommand{\s}{\mathcal{S}}
\newcommand{\A}{\mathcal{A}}
\newcommand{\G}{\mathcal{G}}
\newcommand{\e}{\mathcal{E}}

\newcommand{\B}{\mathcal{B}}
\newcommand{\m}{\mathcal{M}}
\newcommand{\po}{{\partial\Omega}}
\newcommand{\oo}{\mathcal{O}}
\newcommand{\RRR}{\mathcal{R}}

\newcommand{\hm}{\omega}

\newcommand{\RR}{{\mathbb{R}}}

\newcommand{\bp}{\noindent {\it Proof}.\,\,}
\newcommand{\ep}{\hfill$\Box$ \vskip 0.08in}

\DeclareMathOperator{\diam}{diam}
\DeclareMathOperator{\essinf}{essinf}
\def\div{\mathop{\operatorname{div}}}

\begin{document}

\title[$S/N$ estimates and the Dirichlet problem]{Square function/Non-tangential maximal function
estimates and the Dirichlet problem for
non-symmetric elliptic operators}

\author{Steve Hofmann}

\address{Steve Hofmann
\\
Department of Mathematics
\\
University of Missouri
\\
Columbia, MO 65211, USA} \email{hofmanns@missouri.edu}

\author{Carlos Kenig}

\address{Carlos Kenig
\\
Department of Mathematics
\\
University of Chicago
\\
Chicago, IL,  60637 USA} \email{cek@math.chicago.edu}

\author{Svitlana Mayboroda}

\address{Svitlana Mayboroda
\\
School of Mathematics
\\
University of Minnesota
\\
Minneapolis, MN, 55455  USA} \email{svitlana@math.umn.edu}

\author{Jill Pipher}

\address{Jill Pipher
\\
Department of Mathematics
\\
Brown University
\\
Providence, RI,   USA} \email{jpipher@math.brown.edu}

\thanks{Each of the authors was supported by NSF.  
\\
\indent
This work has been possible thanks to the support and hospitality of the \textit{University of Chicago},  the \textit{University of Minnesota}, the \textit{University of Missouri},  \textit{Brown University},  
and the \textit{BIRS Centre in Banff} (Canada). The authors would like to express their gratitude to these institutions.}

\date{\today}


\begin{abstract}
We consider divergence form elliptic operators
$L= -\div A(x) \nabla$, defined in the half space $\reu$, $n\geq 2$,
where the coefficient matrix $A(x)$
is bounded, measurable, uniformly elliptic, $t$-independent, and not necessarily symmetric.
We establish square function/non-tangential maximal function estimates for solutions of 
the homogeneous equation $Lu=0$, and we then combine these estimates with the method of
``$\epp$-approximability" to show that $L$-harmonic measure is absolutely continuous with
respect to surface measure (i.e., n-dimensional Lebesgue measure) on the boundary, 
in a scale-invariant sense: 
more precisely, that it belongs to the class $A_\infty$ with respect to surface measure
(equivalently, that the Dirichlet problem is solvable with data in $L^p$,
for some $p<\infty$).   Previously, these results had been known only in the case $n=1$.
\end{abstract}

\maketitle

\tableofcontents

\section{Introduction and statements of results}

We consider a divergence form elliptic operator
$$L:=-\dv A(x)\nabla,$$
defined in $\mathbb{R}^{n+1}$, 
where $A$ is $(n+1)\times(n+1)$, real, 
$L^\infty$, $t$-independent,  possibly non-symmetric,
and satisfies the 
uniform ellipticity condition
\begin{equation}
\label{eq1.1*} \lambda|\xi|^{2}\leq\,\langle A(x)\xi,\xi\rangle
:= \sum_{i,j=1}^{n+1}A_{ij}(x)\xi_{j} \xi_{i}, \quad
  \Vert A\Vert_{L^{\infty}(\mathbb{R}^{n})}\leq\lambda^{-1},
\end{equation}
 for some $\lambda>0$, and for all $\xi\in\ree$, $x\in\mathbb{R}^{n}$. 
As usual, the divergence form equation is interpreted in the weak sense, i.e., we say that $Lu=0$
in a domain $\Omega$ if $u\in W^{1,2}_{loc}(\Omega)$ and 
$$\int A \nabla u \cdot \nabla \Psi = 0\,,$$ for all  $\Psi \in C_0^\infty(\Omega)$.   
For us, $\Omega$ will be a Lipschitz graph domain
\begin{equation}\label{eq1.2*}
\Omega_\psi:=\{(x,t)\in\ree: t>\psi(x)\}\,,
\end{equation}
where $\psi:\rn\to\RR$ is a Lipschitz function, 
or more specifically (but without loss of generality), $\Omega$ will be the half-space
$\reu:=\{(x,t)\in\mathbb{R}^{n}\times(0,\infty)\}$.

The purpose of this paper is two-fold.

First, we shall establish global and local $L^p$ bounds for the square function
\begin{equation}\label{eq1.square}
S^\alpha(u)(x):=\left(\iint_{|x-y|<\alpha t}|\nabla u(y,t)|^2 \frac{dy dt}{t^{n-1}}\right)^{1/2}\,,
\end{equation}
in terms of the non-tangential maximal function
\begin{equation}\label{eq1.nt}
N^\alpha_*(u)(x):=\sup_{(y,t): |x-y|<\alpha t}|u(y,t)|\,
\end{equation}
(for the sake of brevity we shall refer to such bounds as ``$S<N$" estimates),
and vice versa (we designate these as ``$N<S$" estimates)\footnote{We note that 
in the sequel, when the value of the aperture $\alpha$ is unimportant, or is clear in context,
we shall often simply write $S$ and $N_*$ in lieu of $S^\alpha$ and $N_*^\alpha$.
It is well known that $L^p$ norms for $N_*^\alpha f$
are equivalent for any choice of $\alpha$, and similarly for
$S^\alpha f$ (see \cite{FS}, \cite{CMS}.)}.  As regards the latter, 
we recall that global $N<S$ bounds were already known \cite{AA}; our new contribution here
is to prove a local version.  On the other hand, our $S<N$ estimates are completely new,
for all $n\geq 2$  (the case $n=1$ appeared previously in \cite{KKPT}).

Second, having established (local) $S/N$ estimates, we then use these,
along with the method of ``$\epp$-approximability", to obtain absolute continuity
of $L$-harmonic measure $\hm$
with respect to ``surface" measure $dx$, on the boundary of $\reu$.   In fact, we prove
a stronger, scale-invariant version of absolute continuity, namely that
$\hm$ belongs to the class $A_\infty$.    Let us recall that the latter notion
is defined as follows.   In the sequel, $Q$ will denote a cube in $\rn$.

\begin{definition} ( $A_\infty$, $A_\infty(Q_0)$).  \label{def1.ainfty}  
A non-negative Borel measure $\hm$ defined on $\rn$ (resp., on a fixed cube $Q_0$)
is said to belong to the class $A_\infty$ (resp. $A_\infty(Q_0)$), 
if there are positive constants $C$ and $\theta$
such that for every cube $Q$ (resp. every cube $Q\subseteq Q_0$),
and every Borel set $F\subset Q$, we have
\begin{equation}\label{eq1.ainfty}
\hm (F)\leq C \left(\frac{|F|}{|Q|}\right)^\theta\,\hm (Q).
\end{equation}
\end{definition}
It is well known (see \cite{CF}) that the $A_\infty$ property is equivalent to the condition
that $\hm$ is absolutely continuous with respect to Lebesgue measure, and that
there is an exponent $q>1$ such that the Radon-Nykodym derivative $k:=d\hm/dx$ satisfies 
the ``reverse H\"older" estimate
$$\left(\fint_Qk(x)^q dx\right)^{1/q} \leq C \fint_Q k(x)\, dx\,,$$
uniformly for every cube $Q$ (resp. every $Q\subseteq Q_0$).

It is also well known (see \cite[Theorem 1.7.3]{Ke}) that
the fact that harmonic measure belongs to the class $A_\infty$
is equivalent to the solvability of the following $L^p$ Dirichlet problem, for some $p<\infty$
(in fact for $p$ dual to the exponent $q$ in the reverse H\"older inequality):
\begin{equation}
\begin{cases} Lu=0\text{ in }\mathbb{R}_{+}^{n+1}\\ 
\lim_{t\to 0}u(\cdot,t)=f\text{ in }
L^{p}(\mathbb{R}^{n}) \text{ and } {\rm n.t.}\\ 
\|N_*(u)\|_{L^p(\mathbb{R}^n)}<\infty\,.
\end{cases}\tag{$D_p$}\label{Dp}\end{equation}
Here, the notation ``$u \to f \, {\rm n.t.}$" means that
$\lim_{(y,t) \to (x,0)} u(y,t) = f(x),$ for $a.e. \, x \in \mathbb{R}^n$,
where the limit runs over $(y,t)\, \in \Gamma(x):
=\{(y,t)\in\mathbb{R}_{+}^{n+1}:|y-x|<t\}$.

We also remark that we obtain, as another immediate corollary of the $A_\infty$
property of harmonic measure, that  the layer potentials associated to the operator $L$, as well as
its complex perturbations, enjoy $L^2$ estimates (\cite{H}, \cite{AAAHK}).

We now state our results precisely.  In the sequel, our ambient space will always be $\ree$, with
$n\geq 2$.
\begin{theorem}\label{th1}
Let $L$ be an elliptic operator as above, defined in $\ree$, with $t$-independent coefficients, and suppose that
$Lu=0$ in $\reu$.  Then 
\begin{equation}\label{eq1.SN}
\|S(u)\|_{L^p(\rn)} \lesssim \|N_*(u)\|_{L^p(\rn)}\,,\qquad 0<p<\infty\,,
\end{equation}
where the implicit constant depends upon $p, n$, ellipticity, and the apertures of the cones defining
$S$ and $N_*$.
\end{theorem}

The previous theorem has the following immediate local corollary.  
Given a cube $Q\subset \rn$, let 
\begin{equation}\label{eq1.cbox}
T_Q:= Q\times \Big(0,\ell(Q)\Big)\subset\reu
\end{equation}
denote the standard  
Carleson box  above $Q$, where, here and in the sequel, $\ell(Q)$
is the side length of $Q$.  
\begin{corollary}\label{c1}
Under the same hypotheses as in Theorem \ref{th1}, for a bounded solution $u$,
we have the Carleson measure estimate
\begin{equation}\label{eq1.9}
\sup_Q \frac1{|Q|} \iint_{T_Q} |\nabla u(x,t)|^2 t dt dx \leq C\, \|u\|_{L^\infty(\Omega)}\,,
\end{equation}
where $C$ depends only upon dimension and ellipticity.
\end{corollary}

\begin{proof}[Sketch of proof of Corollary \ref{c1}]
The corollary may be deduced from the theorem by a variant of the argument in \cite{FS}: 
we divide the boundary data
into  a ``local" part plus a  ``far-away" part (which we further sub-divide 
in a dyadic annular fashion), and then use Theorem \ref{th1} to handle the local part,
and H\"older continuity at the boundary to obtain summable decay for the dyadic terms in the 
far-away part.  The treatment of the local part requires in addition the 
use of a decay estimate for solutions with boundary data vanishing outside a 
cube (cf. Lemma \ref{l4.8} below).  We omit the details.
Alternatively,  \eqref{eq1.9} may be gleaned directly from local estimates established in our proof of Theorem \ref{th1} (cf. Section \ref{sSN} below, where we shall make note of the local estimates in
question, during the course of the proof).  
\end{proof}

We recall that 
the converse direction to Theorem \ref{th1}, at least in the case $p=2$,
has recently been obtained by
Auscher and Axelsson, and appears in \cite[Theorem 2.4, part (i)]{AA}, as follows:
\begin{equation}\label{eq1.NS}
\|N_*(u)\|_{L^2(\rn)} \lesssim \|S(u)\|_{L^2(\rn)}\,.
\end{equation}
In fact, the result of \cite{AA} is considerably more general, in that \eqref{eq1.NS} holds
in the case of complex coefficients and even strongly elliptic systems, and furthermore
the hypothesis of $t$-independence may be relaxed to a sort of scale-invariant square Dini
smoothness in the $t$-variable, averaged in $x$.  We refer the reader to \cite{AA} for details.
We remark that it is still an (apparently difficult) open problem to
extend Theorem \ref{th1} (that is, the $S<N$ direction), to the case of complex coefficients,
even assuming $t$-independence as we do here.   

With \eqref{eq1.NS}, the global estimate of \cite{AA}, in hand, we shall deduce a local version.
Given a cube $Q\subset \rn$, let $\theta Q$ denote the concentric cube of side length $\theta\,\ell(Q)$,
and let
\begin{equation}\label{eq1.cboxshort}
R_Q:=Q\times \Big(0,\ell(Q)/2\Big)\,,
\end{equation}
be the ``short" Carleson box above $Q$.

\begin{theorem}\label{th2} Let $L$ be a $t$-independent elliptic operator as above,  
and suppose that
$u\in L^\infty$ is a solution of
$Lu=0$ in $\reu$.
Then for each cube
$Q\subset \rn$, and each $0<\theta<1$, 
there is a set $K_Q= K_Q(\theta)\subset\subset R_Q$, with
$\dist(K_Q,\partial R_Q)\approx \ell(Q)$ (depending upon $\theta$), such that
\begin{equation}\label{eq1.NSloc}
\fint_{\theta Q} |u(x)|^2\, dx\leq C_{\theta} 
\left(\frac1{|Q|} \iint_{R_Q} |\nabla u(x,t)|^2 t dt dx \,+\,\sup_{K_{Q}} |u|^2 \right)\,,
\end{equation}
where the constant $C_{\theta}$ 
depends also on dimension and ellipticity.
\end{theorem}

\begin{remark} \label{remark1.16}
We note that our proof of Theorem \ref{th2} (cf. Section \ref{sNS} below)
will actually show something stronger, namely, that \eqref{eq1.NSloc} holds with the left hand side replaced by $\fint_{\theta Q} N_{*,Q}(u)^2\, dx $, where $N_{*,Q}$ is a truncated non-tangential
maximal operator, defined with respect to cones that have been truncated at height $\approx \ell(Q)$.
\end{remark}

We note that Theorem \ref{th1}, the global $N<S$ bound \eqref{eq1.NS},
and Theorem \ref{th2},  imply 
generalizations of themselves.    These respective generalizations may be summarized as follows.
\begin{corollary}\label{c2} Let $L$ be as above, let $\Omega_\psi$ be a Lipschitz graph domain (cf. \eqref{eq1.2*}), and suppose that $Lu=0$ in $\Omega_\psi$.  Then for every $p\in (0,\infty)$, we have
\begin{equation}\label{eq1.10}
\int_{\po_\psi}S_\psi(u)^p d\sigma \lesssim \int_{\po_\psi}N_{*,\psi}(u)^p d\sigma\,,
\end{equation}
and
\begin{equation}\label{eq1.11}
\int_{\po_\psi}N_{*,\psi}(u)^p d\sigma \lesssim \int_{\po_\psi}S_\psi(u)^p d\sigma\,,
\end{equation}
where the implicit constants depend upon $n$, $p$, ellipticity, and $\|\nabla\psi\|_\infty$.
Moreover, for $0<\theta<1$, if $0\leq\psi(x)\leq \ell(Q)/8$ in $Q$, 
and if $u \in L^\infty$ is a solution of $Lu=0$ in $\Omega_\psi$, 
then there is a set
$K^\psi_Q=K_Q^\psi(\theta)\subset\subset T_Q\cap\Omega_\psi$, with 
$\dist\big(K^\psi_Q, \partial(T_Q\cap\Omega_\psi)\big)\approx\ell(Q)$ (depending on $\theta$ and $\|\nabla \psi\|_\infty$), such that
\begin{multline}\label{eq1.12}
\fint_{\theta Q}|u(x,\psi(x))|^2 dx \\[4pt]
\leq C_{\theta}
\left( \frac1{|Q|}\iint_{T_Q\cap\Omega_\psi} |\nabla u(x,t)|^2 \,\Big(t-\psi(x)\Big)\, dt dx 
\,+\,\sup_{K^\psi_Q} |u|^2 \right)\,,
\end{multline}
where $C_{\theta}$ depends also 
upon $n$, ellipticity, and the Lipschitz constant of $\psi$.
\end{corollary}

Here, $d\sigma = d\sigma(x):= \sqrt{1+|\nabla\psi(x)|^2} dx \approx dx$ 
denotes the standard surface measure on the Lipschitz graph
$\po_\psi$.    The square function $S_\psi(u)$ and non-tangential maximal function $N_{*,\psi}(u)$ 
are defined on $\Omega_\psi$ as follows:
\begin{equation}\label{eq1.squarelip}
S_\psi(u)(x):=\left(\iint_{\Gamma(x)}|\nabla u(Y)|^2 \frac{dY}{\delta(Y)^{n-1}}\right)^{1/2}\,,
\end{equation}
\begin{equation}\label{eq1.ntlip}
N_{*,\psi}(u)(x):=\sup_{\Gamma(x)}|u(Y)|\,,
\end{equation}
where $\delta(Y):= \dist(Y,\po_\psi)$, and where $\Gamma(x)\subset\Omega_\psi$ is a vertical cone
with vertex at $x\in\po_\psi$, of sufficiently narrow aperture (depending upon
the Lipschitz constant of  $\psi$) that $\delta(Y)\approx |Y-x|,\, \forall Y\in \Gamma(x)$.

\begin{proof}[Sketch of proof of Corollary \ref{c2}]  Since 
Theorem \ref{th1}, Theorem \ref{th2}, and \eqref{eq1.NS} hold (or will be shown to hold),
for the entire class of $t$-independent divergence form operators as described above, 
one may reduce matters to the case
that $\psi\equiv 0$ (i.e., the case that $\Omega_\psi =\reu$), by  
``pulling back" under the mapping
$(x,t)\to(x,t+\psi(x))$, which preserves the class of $t$-independent elliptic operators under consideration, and maps $\Omega_\psi\to\reu$, and $\po_{\psi}\to \partial\reu$, bijectively.
In the case of \eqref{eq1.11}, the pullback mechanism and \eqref {eq1.NS}
yield directly only the case $p=2$;  however, since we also establish local ``$N<S$"
estimates (cf. Remark \ref{remark1.16}), we may obtain the full range of $p$ in \eqref{eq1.11} by a well known ``good-lambda" argument.
We omit the details, which are standard.
\end{proof}

Using the local estimates \eqref{eq1.9} and \eqref{eq1.12}, we shall deduce the following
theorem.
Given a cube $Q\in\rn$, we let $x_Q$ denote the center of $Q$, and let
$X_Q:=(x_Q,\ell(Q))$ be the ``Corkscrew point" relative to $Q$.  For $X\in \reu$, 
and an elliptic operator $L$ as above, we
let  $\hm^{X}$ denote the $L$-harmonic measure at $X$.
\begin{theorem}\label{th3} Let $L$ be a divergence form elliptic operator as above, with
$t$-independent coefficients.  Then there is a $p<\infty$ such that the Dirichlet problem
$D_p$ is well-posed; equivalently, for each cube $Q\subset \rn$, the $L$-harmonic measure
$\hm^{X_Q} \in A_\infty(Q)$, with constants that are uniform in $Q$.  
\end{theorem}

The proof of Theorem \ref{th3} will be deduced from \eqref{eq1.9} and \eqref{eq1.12}
via the method of ``$\epp$-approximability".  We defer until Section \ref{s-epp-app} a detailed discussion
of this notion, but we mention at this point that
it was introduced by Garnett \cite{G}, who showed that the property is enjoyed by 
bounded harmonic functions in the plane.  An alternative proof of Garnett's result was provided by Varopoulos \cite{V}.  A third proof,
which extended to bounded harmonic functions in all dimensions, was found by Dahlberg \cite{D},
who made the connection with square function estimates
on bounded Lipschitz domains.
In \cite{KKPT}, it was observed by the second and fourth named authors of this paper, 
jointly with Koch and Toro, that Dahlberg's argument may be carried over to 
bounded solutions of general divergence form elliptic operators, in the presence of
square function estimates on bounded Lipschitz domains;  moreover, these authors 
showed that $\epp$-approximability, in turn, implies that harmonic measure belongs to $A_\infty$ with respect to surface measure on the boundary.  In the present paper, we invoke the latter result of
\cite{KKPT} ``off-the -shelf":  the essence of the proof of our Theorem \ref{th3} is to
show that our solutions are $\epp$-approximable.  Having done this (in Section
\ref{s-epp-app}),  we then obtain immediately the conclusion
of Theorem \ref{th3}, by \cite[Theorem 2.3]{KKPT}.
We remark that
our approach here, although it relies upon ideas from the proofs in both \cite{G} and \cite{D}, 
does {\it not}, in contrast to the proofs of $\epp$-approximability in \cite{D} and \cite{KKPT}, 
require $S/N$ estimates on  Lipschitz sub-domains of arbitrary orientation, but rather only  local $S/N$ estimates on Lipschitz graph domains $\Omega_\psi$ as in
\eqref{eq1.2*}, for which the fixed vertical (i.e., $t$) direction is transverse to $\po_\psi$.  This refinement of the $\epp$-approximability method is significant for us, because it is not clear how 
(or whether) one could exploit the $t$-independence of our coefficients to obtain $S/N$ estimates
on Lipschitz domains with other orientations (i.e., for which the $t$-direction may fail to be 
transverse to the boundary).

Finally, we note that, by \cite{H} and \cite{AAAHK}, 
Theorem \ref{th3} has as an immediate corollary that the layer potentials
associated to any $t$-independent operator $L$ as above, and to its complex perturbations,
are $L^2$ bounded.  More precisely, let $\mathcal{E}_L(x,t,y,s)$ be the fundamental solution for $L$,
and  define the single layer potential operator by 
 \begin{equation}
\label{eq1.5}S^L_{t}f(x) :=\int_{\mathbb{R}^{n}}\mathcal{E}_L(x,t,y,0)\,f(y)\,dy, \,\,\, t\in \mathbb{R}
\end{equation}
\begin{corollary}\label{c3} Let $L=-\div A(x) \nabla$ be a $t$-independent divergence form
elliptic operator, where $A$ is real, or more generally, where $A$ has complex entries and 
there is a real, elliptic, $t$-independent matrix $A'(x)$ such that $\|A-A'\|_{L^\infty(\rn)}<\eps_0$.
If $\eps_0$ is small enough, depending only upon dimension and ellipticity, then
$$\sup_{t> 0}\int_{\rn}|\nabla_{x,t} S_t^L f(x)|^2\, dx  
+\iint_{\reu} |\nabla_{x,t}\partial_t S_t^L f(x)|^2\, \frac{ dx dt}{t} \leq C \int_{\rn} |f(x)|^2\, dx\,,$$
where $C$ depends upon $n$, ellipticity, and $\|A-A'\|_{L^\infty(\rn)}$.
\end{corollary}
The case that $A$ has real entries follows immediately from Theorem \ref{th3} and 
\cite[Theorem 3.1 and its proof]{H}.  
In turn, the perturbation result follows from the proof of \cite[Theorem 1.12]{AAAHK}, plus the global $N<S$ bound of \cite{AA} (that is, \eqref{eq1.NS} above).  We omit the details.

\subsection{Historical comments, and remarks on the proofs of the theorems.}
In the case of $t$-independent symmetric matrices, all of the results stated above have been known for some time.  In that case, solvability of the Dirichlet problem $D_2$ was proved in \cite{JK}, by means of a so-called ``Rellich identity" obtained via integration by parts.    In turn, given the solvability result,
$S/N$ bounds follow by the main theorem in \cite{DJK} (thus, for symmetric matrices,
the logic of our proof strategy in the present paper, in which we establish 
$S/N$ bounds first, and then deduce solvability, was reversed).
The integration by parts argument used to prove the Rellich identity 
relies heavily on self-adjointness, and thus is inapplicable to the
non-symmetric case treated here.  Let us further point out that self-adjointness plays another role:
in the case of real symmetric coefficients, one obtains $L^2$ solvability of the Dirichlet problem
(equivalently, that the Poisson kernel satisfies a reverse H\"older inequality with exponent $q=2$),
whereas in the case of non-symmetric coefficients, by the counter-examples of
\cite{KKPT}, one cannot make precise the exponent $p$
for which one has solvability of $D_p$ (equivalently, one cannot specify the reverse H\"older
exponent $q$ enjoyed by the Poisson kernel).  Thus, for non-symmetric operators,
the conclusion that $\hm \in A_\infty$ is best possible.

Our main results, Theorems \ref{th1}, \ref{th2} and \ref{th3}, are extensions 
to $\reu$, $n\geq 2$, of analogous results of \cite{KKPT}, which were valid in the plane
(i.e., $n=1$).  The proof of Theorem \ref{th2} will follow that of its antecedent,
Theorem 3.18 of \cite{KKPT}, very closely, with some minor changes required by
the higher dimensional setting.    As noted above, the proof of Theorem \ref{th3} is based on
the ``$\epp$-approximability" arguments of \cite{G}, \cite{D} and \cite{KKPT}, 
in which $S/N$ estimates on Lipschitz sub-domains is used to obtain a certain 
approximability property of solutions, and in turn, to deduce solvability of $D_p$ for some finite $p$. 
In this paper, we present a
non-trivial refinement of the method, which requires us to establish (local) comparability of $S$ and $N$
only on Lipschitz graph domains, for which the $t$-direction is transverse to the boundary.

The  $S<N$ estimates proved in \cite{KKPT} relied
on the fact that in the plane, a $2\times 2$ $t$-independent matrix can be 
triangularized by ``pushing forward" to an appropriate Lipschitz graph domain $\Omega_1$.
In turn, one can prove square function estimates 
for operators with upper triangular coefficient matrices,
by a standard integration by parts argument,
since for such operators, the function $v(x,t) \equiv t$ is an adjoint null solution.
Having triangularized the matrix, this integration by parts may be carried out in the half-plane 
$\RR^2_+$, and even in Lipschitz graph domains, after ``pulling back" to the half-space
with the Dahlberg-Kenig-Stein change of variable.  

In higher dimensions, this approach fails, but the
proof of Theorem \ref{th1} exploits a more general principle in the same spirit, namely,
that by pushing forward to the domain above the graph of an appropriate
$W^{1,2+\eps}$ function $\vp$, which arises in a (local) $L$-adapted
Hodge decomposition of
the coefficient vector ${\bf c}:=(A_{n+1,j})_{1\leq j\leq n}$, one may put the coefficient
matrix into a better form, in which the vector ${\bf c}$ is replaced by
a divergence free vector.
In turn, this observation may be combined with an $L$-adapted variant of the
Dahlberg-Kenig-Stein pullback mapping, along with the solution of the Kato problem
\cite{HLMc}, \cite{AHLMcT}, 
to carry out a refined version of the classical integration by parts argument.  Of course, some care must be taken with the push forward/pullback mapping based on $\vp$, since the latter is merely
$W^{1,2+\eps}$, and not Lipschitz.

\subsection{Notation}
In the sequel, we shall use the notational convention that a generic constant $C$, 
as well as the  constants implicit in the expressions $a\lesssim b, a\approx b, a\gtrsim b$,
shall be allowed to depend on dimension, ellipticity, the aperture of the cones used
in the definition of $S$ and $N_*$ (with one exception, to be noted momentarily), 
and, when working in Lipschitz graph domains,
the Lipschitz constant, 
unless there is an explicit qualification to the contrary.    As regards constants depending
on the aperture of the cones, in ``Step 2" of the proof of Theorem \ref{th1}, we shall
consider non-tangential maximal functions taken with respect to a narrow aperture $\eta$,
and we shall indicate explicitly any dependence on $\eta$,
of the norms of these maximal functions (thus, if no dependence on $\eta$ is indicated,
there is none, or we have reached a stage of the argument where such dependence is irrelevant;  cf. \eqref{eq2.20}-\eqref{eq2.21} and Subsection \ref{ss3.2} below.)
We shall sometimes write $X=(x,t)$ to denote points in $\ree$,
and we let $B(X_0,r):=\{X\in\ree: |X-X_0|<r\}$ denote the standard
Euclidean ball in $\ree$.  We shall denote cubes in $\rn$ and in $\ree$, respectively,
by $Q\subset \rn$ and $I\subset \ree$.

\section{Proof of Theorem \ref{th1}:  Preliminaries for ``$S<N$"}

Let $A(x)$ be an $(n+1)\times(n+1)$, real, elliptic,
$L^\infty$, $t$-independent and possibly non-symmetric matrix, as in the introduction. 
We represent the matrix $A$ schematically as follows:

\begin{equation}
A= \left[\begin{array}{c|c}
 A_\| & {\bf b}\\[4pt]
\hline {\bf c} & d\end{array}\right],\label{eq1.1}\end{equation}
where $d:= A_{n+1,n+1}$, ${\bf b}:= (A_{i,n+1})_{1\leq i\leq n},\,{\bf c}:= (A_{n+1,j})_{1\leq j\leq n},$
and $A_\|$ denotes the $n\times n$ submatrix of $A$ with entries $(A_\|)_{i,j} := A_{i,j},\,1\leq j\leq n$.
Given any matrix $B = (B_{i,j})$ (no matter its dimensions), we let $B^* = (B_{j,i})$ denotes its 
adjoint (i.e. transpose, since our coefficients are real).  Thus,
\begin{equation}
A^*= \left[\begin{array}{c|c}
 A_\|^* & {\bf c}\\[4pt]
\hline {\bf b} & d\end{array}\right].\label{eq1.2}\end{equation}

Eventually, we shall establish ``good-lambda" estimates for square functions
of solutions of the equation $Lu=0$, and thus, as usual, we shall work locally, on a 
given cube $Q\subset\rn$.  
Since our coefficients clearly belong to $L^p_{loc}$ for any finite $p$, having fixed a cube $Q$,
we can make a $W^{1,2+\eps}$ Hodge decomposition 
with sufficiently small $\eps>0$ (see, e.g., \cite{AT}), 
and write
\begin{equation}\label{eq1.hodge}
{\bf c}1_{5Q} = - A_\|^* \nabla \varphi + {\bf h}, \qquad {\bf b}1_{5Q}
=  A_\| \nabla \tilde{\varphi} + \widetilde{\bf h}\,,
\end{equation}
where $\varphi,\tilde{\vp}\in  W_0^{1,2+\eps}(5Q)$, and
${\bf h},\widetilde{\bf h}$ are divergence free and supported in $5Q$,
and where
\begin{align}\label{eq1.hodge*}
\fint_{5Q}\Big( |\nabla \varphi(x)| 
+|{\bf h}(x)|\Big)^{2+\eps}dx \,\leq\, C\fint_{5Q} |{\bf c}(x)|^{2+\eps} dx\, &\leq C\\[4pt]\label{eq1.hodge**}
\fint_{5Q}\Big( |\nabla \tilde{\varphi}(x)|
+|\widetilde{\bf h}(x)|\Big)^{2+\eps} dx\, \leq\, C\fint_{5Q} |{\bf b}(x)|^{2+\eps} dx\, &\leq C\,.
\end{align}

We define an $n$-dimensional divergence form operator 
$$L_\|:= -\dv_x (A_\| \nabla_x)\,,$$
and let $\p_t:= e^{-t^2 L_\|}$ and $\p^*_t:= e^{-t^2 L^*_\|}$ denote, respectively,
the heat semigroup associated to $L_\|$ and to its adjoint $L_\|^*$, but endowed
with ``elliptic" homogeneity (thus,  $t$ has been squared).

In the sequel,
we shall want to consider the pullback of $L$ under the mapping
\begin{equation}\label{eq2.6}
\rho(x,t):= \big(x,\tau(x,t)\big):= \left(x,t-\vp(x)+\pe\vp(x)\right)\,,
\end{equation}
where $\eta>0$ is a small but fixed number to be chosen, and $\vp$ is as in
\eqref{eq1.hodge}, and has been extended to all of $\rn$ by setting $\vp\equiv 0$ in $\rn\setminus 5Q$.
A computation shows that if $u$ is a solution of $Lu=0$ in $\reu$,
then $u_1 := u\circ \rho$ is a solution of $L_1 u_1=0$
(at least formally),
where $L_1:= -\dv(A_1\nabla)$, and, for $J$ and $\bf{p}$ to be defined momentarily,

\begin{equation}
A_1:= \left[\begin{array}{c|c}
 J \, A_\| & {\bf b} +A_\| \nabla_x \varphi-A_\|\nabla_x \pe\vp\\[7pt]
\hline\\ {\bf h}-A^*_\| \nabla_x \pe\vp& \frac{\langle A \,{\bf p},{\bf p}\rangle}J 
\end{array}\right].\label{eq2.2}
\end{equation}

\noindent
Here, ${\bf h}$ is the divergence free vector in the Hodge decomposition
\eqref{eq1.hodge}, and we define $J$ and {\bf p} as follows:
\begin{equation}\label{eq2.jdef}
J(x,t):= 1+\partial_t \pe\vp(x)\,,
\end{equation}
is the Jacobian of the change of variable $t\to \tau(x,t)$, with $x\in\rn$ fixed,
and
\begin{equation}\label{eq2.pdef}
{\bf p}(x,t):= (\nabla_x \tau(x,t),-1)=(\nabla_x\pe\vp(x)-\nabla_x\vp(x),-1)\,.
\end{equation}
Let us make precise our statement that $L_1u_1=0$. 
In fact, in the sequel, we shall 
consider $u_1$ in a
certain sawtooth domain $\Omega_0$ in which the mapping $(x,t)\to\rho(x,t)$ is 1-1, with range contained in $\reu$, and in which
$J(x,t) \approx 1$ (uniformly).
The fact that $L_1u_1=0$ in the sawtooth region then follows from the pointwise identity
\begin{equation}\label{eq2.10}
 A\Big((\nabla u)\circ \rho\Big)\cdot\Big((\nabla v)\circ \rho\Big) \, J = A_1\nabla u_1\cdot \nabla v_1\,,
\end{equation}
for $v \in W^{1,2}(\Omega_0)$, where $v_1:= v\circ \rho$.

We conclude these preliminaries with some estimate for square functions 
and non-tangential maximal functions built from
the ``ellipticized" heat semigroup operators $\p_t$ and $\p_t^*$.
By the solution of the Kato problem \cite{HLMc}, \cite{AHLMcT}, 
we have for every $\alpha >0$ that
\begin{multline}\label{eq1.kato}
\int_{\rn} \iint_{|x-y|<\alpha t}|t \,\p_t\dv_x {\bf f}(y)|^2 \frac{dy dt}{t^{n+1}}dx\\[4pt]\approx\,
\iint_{\reu}|t \,\p_t\dv_x {\bf f}(x)|^2 \frac{dx dt}{t}\, \leq\, C \,\|{\bf f}\|_{L^2(\rn)}^2 \,,
\end{multline}
where the implicit constants depend upon the aperture $\alpha$ (but in fact are uniform
for all $\alpha\leq 1$.
Also, by standard semigroup theory (more precisely,  that
$\p_t = e^{-t^2L_\|/2}e^{-t^2L\|/2}$, and that $t\nabla_{x,t}e^{-t^2L_\|/2}$ is bounded on $L^2(\rn)$, uniformly in $t$;  cf. \cite{Ka}), the latter bounds
imply that 
\begin{multline}\label{eq1.kato*}
\int_{\rn} \iint_{|x-y|<\alpha t}|t^2 \nabla_{x,t}\,\p_t\dv_x {\bf f}(y)|^2 \frac{dy dt}{t^{n+1}}dx\\[4pt]\approx\,
\iint_{\reu}|t^2\nabla_{x,t} \,\p_t\dv_x {\bf f}(x)|^2 \frac{dx dt}{t}\, \leq\, C \,\|{\bf f}\|_{L^2(\rn)}^2 \,.
\end{multline}
Of course, analogous bounds hold for $\p_t^*$.   By a well-known argument of Fefferman and Stein
\cite{FS}, the bounds in \eqref{eq1.kato}-\eqref{eq1.kato*} imply corresponding Carleson 
measure estimates when ${\bf f}\in L^\infty(\rn)$, and thus by tent space interpolation 
\cite{CMS}, we obtain that
\begin{equation}\label{eq1.katop}
\|\A^\alpha_1 {\bf f}\|_{L^p(\rn)}+ \|\A^\alpha_2 {\bf f}\|_{L^p(\rn)}  \leq\, C_{\alpha,p} 
\,\|{\bf f}\|_{L^p(\rn)} \,,
\end{equation}
for every $p\in [2,\infty)$, where
\begin{align}\label{eq1.adef1}
\A^\alpha_1 {\bf f}(x)&\, :=\,
\left( \iint_{|x-y|<\alpha t}|t \,\p_t\dv_x {\bf f}(y)|^2 \frac{dy dt}{t^{n+1}}\right)^{1/2}\\[4pt]\label{eq1.adef2}
\A^\alpha_2 {\bf f}(x)&\, :=\,
\left( \iint_{|x-y|<\alpha t}|t^2 \nabla_{x,t}\,\p_t\dv_x {\bf f}(y)|^2 \frac{dy dt}{t^{n+1}}\right)^{1/2}\,\,.
\end{align}

Trivially, \eqref{eq1.kato}-\eqref{eq1.kato*} also entail $L^2$ bounds for the vertical square functions
\begin{align}\label{eq1.gdef1}
\G_1 {\bf f}(x)&\, :=\,
\left( \int_0^\infty |t\,\p_t\dv_x {\bf f}(x)|^2 \frac{dt}{t}\right)^{1/2}\\[4pt]\label{eq1.gdef2} 
\G_2 {\bf f}(x)&\, :=\,
\left( \int_0^\infty |t^2 \nabla_{x,t}\,\p_t\dv_x {\bf f}(x)|^2 \frac{dt}{t}\right)^{1/2}\,.
\end{align}
The $L^2$ bounds for these vertical square functions may also be extended to $L^p$:
\begin{equation}\label{eq1.katovertp}
\|\G_1 {\bf f}\|_{L^p(\rn)}+ \|\G_2 {\bf f}\|_{L^p(\rn)}  \leq\, C_{p} 
\,\|{\bf f}\|_{L^p(\rn)} \,,
\end{equation}
for every $p\in [2,2+\eps_0)$, with $\eps_0>0$ chosen small enough
 depending on dimension and ellipticity.  
For $\G_1$ the latter fact is a routine consequence of
local H\"older regularity in $x$, of the kernel of $\p_t$, and in fact the $L^p$ bounds hold more generally
for $2\leq p<\infty$;  for $\G_2$, the $L^p$ estimates in the range $2<p< 2+\eps_0$ are essentially due
to Auscher \cite{A}, and in that case the upper endpoint $2+\eps_0$ is best possible.

Clearly,  \eqref{eq1.katop} and \eqref{eq1.katovertp} hold also 
for the analogous operators 
corresponding to $\p^*_t$.

Finally, we note that for $2\leq p< \infty$, 
\begin{align}\label{eq2.19}
\|N^\alpha_*(\partial_t \p_t f)\|_p 
&\,\leq\, C_{\alpha,p} \,\|\nabla_x f\|_p\\[4pt]\label{eq2.20}
\|\eta^{-1}N^\eta_*(\partial_t \Pe f)\|_p 
&\,\leq\,  \,C_p\, \|\nabla_x f\|_p \\[4pt]\label{eq2.21}
\|\widetilde{N}^\eta_*(\nabla_x \Pe f)\|_p
&\,\leq\,  \,C_p\, \|\nabla_x f\|_p 
\end{align}
and similarly for $\p^*_t$, where we shall define
$\widetilde{N}^\eta_*$ momentarily.  
Indeed, since the kernel of the operator $t\partial_t\p_t$ enjoys pointwise
Gaussian bounds, and kills constants, we have 
$$|\partial_t\p_t f(y)|=|\partial_t\p_t (f-f_{x,t})(y)|\leq C_{\alpha} M(\nabla_x f)(x)\,,$$
whenever $|x-y|<\alpha t$, where $f_{x,t}:= 
\fint_{|x-z|<t} f(z) dz$, and where in the last step we have used
a dyadic annular decomposition, the decay of the kernel, a telescoping identity, 
and the $L^1$ Poincare inequality.  
The bound \eqref{eq2.19} now follows immediately.
A slightly more careful version of the same argument,
in which we replace $f_{x,t}$ by $f_{x,\eta t}$, yields \eqref{eq2.20},
since the kernel of $t\partial_t  \Pe$, call it $k_{\eta t}(x,y)$ enjoys the Gaussian estimate
$$|k_{\eta t}(x,y)|\lesssim (\eta t)^{-n} \exp\left(-\frac{|x-y|^2}{\eta^2 t^2}\right)\,.$$
Here, our interest is in the case that $\eta$ is fairly small, so it is important that we have
specified that the aperture
of the cone in \eqref{eq2.20} is equal to $\eta$ (it would of course
also be fine to allow any aperture $\alpha \lesssim \eta$). 
To prove \eqref{eq2.21}, in which
\begin{equation}\label{eq2.22*}
\widetilde{N}^\eta_*(v)(x):= \sup_{(y,t):\,|x-y|<\eta t} \left(\fint_{|y-z|< \eta t} |v(z,t)|^2
dz \right)^{1/2}\,,
\end{equation}
we may argue as in \cite{KP}, using a variant of Caccioppoli's inequality
to obtain a bound in terms of $N_*^{2\eta}(\partial_t \Pe f)$,  $\sup_{t>0}|\partial_t\p_t f|$, and
a tangential gradient on the boundary.    
We omit the details.

 \section{Proof of Theorem \ref{th1}:  Main arguments for ``$S<N$"}\label{sSN}

In this section, we present the main arguments in the proof of Theorem \ref{th1},
in three steps. 
We first show that $S(u)$ is controlled, in $L^p$ norm for $p$ 
sufficiently large, by a vertical square function 
involving only the $t$-derivative of $u$ (plus $N_*(u)$).  We then show that this
vertical square function is controlled by $N_*(u)$, again in $L^p$ norm for $p$ large.
Finally, we shall remove the restriction on $p$.   We will sometimes vary the apertures of our cones,
in the definitions of $S(u)$ and $N_*(u)$, from
one of these steps to the next, but as we have already noted, this is harmless,
as all choices of aperture yield
equivalent $L^p$ norms (\cite{FS}, \cite{CMS}.)  Within each step, 
we shall always maintain a consistent choice of aperture.

 \subsection{Step 1:  $S(u)$ is controlled by a vertical square function of $\partial_t u$.}
 
Set
\begin{equation}\label{eq3.vert}
g(u)(x):= \left(\int_0^\infty |\partial_t u(x,t)|^2\,tdt\right)^{1/2}\,.
\end{equation}
Our goal at this stage is to establish the following ``good-$\lambda$" inequality,
for arbitrary positive $\lambda$, and for all sufficiently small $\gamma$:
\begin{equation}\label{eq3.goodl}
\left|\Big\{x\in Q: S(u)(x)>3\lambda, \left(M\left(g(u)^2 + N_*(u)^2\right)(x)\right)^{1/2}
\leq \gamma \lambda\Big\}\right| \leq\,C\gamma^2 |Q|\,,
\end{equation}
whenever $Q$ is a Whitney cube for the open set $\{S(u)>\lambda\}$.
Here and in the sequel, $M$ denotes the non-centered Hardy-Littlewod maximal operator, taken with respect to averages on cubes.  As is well known,  \eqref{eq3.goodl} implies the global $L^p$ bound 
\begin{equation}\label{eq3.global1}
\|S(u)\|_{L^p(\rn)}\,\leq\, C_p\,\Big(\|g(u)\|_{L^p(\rn)}  +\|N_*(u)\|_{L^p(\rn)}\Big)\,,\qquad 2<p<\infty\,. 
\end{equation}
For the sake of specificity, let us fix the aperture of the cones defining $S(u)$ to be 1, and that of the
cones defining $N^*(u)$ to be $\gg 1$. 

We now fix a cube $Q$ in the Whitney decomposition of $\{S(u)>\lambda\}$, and
we introduce a truncated square function
$$S_Q(u)(x):=\left(\iint_{|x-y|< t<\ell(Q)}|\nabla u(y,t)|^2 \frac{dy dt}{t^{n-1}}\right)^{1/2}\,.$$
To prove \eqref{eq3.goodl}, we may suppose that there is at least one
point in $Q$, call it $x_*$, for which 
\begin{equation}\label{eq3.4++}
\left(M\left(g(u)^2 + N_*(u)^2\right)(x_*)\right)^{1/2}
\leq \gamma \lambda\,.
\end{equation}
Then by the arguments of \cite{DJK}
(which are now standard), using interior estimates for solutions,
properties of Whitney cubes, and the fact that the cones defining $N_*(u)$ have aperture
much larger than do those defining $S(u)$,  the set on the left hand 
side of \eqref{eq3.goodl} is contained in
$\{x\in Q: S_Q(u)(x)>\lambda\}$,
provided $\gamma$ is chosen small enough, depending on dimension and ellipticity.
We omit the details, which may be found in \cite{DJK}.
By Tchebychev's inequality, and then Fubini's Theorem, 
we therefore have that the left hand side of \eqref{eq3.goodl} is bounded by
\begin{multline}\label{eq3.5++}
\left|\Big\{x\in Q: S_Q(u)(x)>\lambda\Big\}\right| \\[4pt]
\leq \frac1{\lambda^2} \int_QS_Q(u)^2(x)\,dx \lesssim  \frac1{\lambda^2} 
\int_{3Q} \int_0^{\ell(Q)} |\nabla u(y,t)|^2 \,t dtdy=: \frac1{\lambda^2}\,I\,.
\end{multline}
We claim that
\begin{equation}\label{eq3.6++}
I\lesssim  |Q| \,M\left(g(u)^2 + N_*(u)^2\right)(x_*)\,,
\end{equation}
whence \eqref{eq3.goodl} follows from \eqref{eq3.4++}.

Let us now verify the claim.  Set $\Phi_Q(t)\equiv \Phi(t/\ell(Q))$, where 
$\Phi \in C^\infty (\RR)$, with $0\leq\Phi\leq 1$,
$\Phi(t)\equiv 1$ if $t\leq  1$, and $\Phi(t) \equiv 0$ if $t\geq 2$.  
Integrating by parts in $t$, we then have that 
\begin{multline*}
I \leq \int_{3Q} \int_0^{2\ell(Q)} |\nabla u(y,t)|^2\,\Phi_Q(t) \,t dtdy
\approx \int_{3Q} \int_0^{2\ell(Q)} \partial_t\left(|\nabla u(y,t)|^2\,\Phi_Q(t) \right)\,t^2 dtdy\\[4pt]\lesssim
\int_{3Q} \int_0^{2\ell(Q)} \Big\langle\nabla\partial_t u(y,t),\nabla u(y,t)\Big\rangle\,\Phi_Q(t) \,t^2 dtdy
\,+\,\int_{3Q} \fint_{\ell(Q)}^{2\ell(Q)} |\nabla u(y,t)|^2 \,t^2 dtdy \\[4pt]
=: I' + I''\,.
\end{multline*}
By Caccioppoli's inequality, $I'' \lesssim |Q|\, M\left(N_*(u)^2\right)(x_*)$. 
Moreover, by Cauchy's inequality, we have that
$$I'\lesssim \epp \int_{3Q} \int_0^{2\ell(Q)} |\nabla u(y,t)|^2\,\Phi_Q(t) \,t dtdy
\,+\, \frac1{\epp}\int_{3Q} \int_0^{2\ell(Q)} |\nabla \partial_t u(y,t)|^2 \,t^3 dtdy\,.$$
Fixing $\epp$ small enough, depending only upon allowable parameters, 
we may hide the first of these terms (to do this rigorously, we would smoothly
truncate the $t$-integral away from 0, to guarantee that $I$ is finite;  the truncation results
in additional error terms which may be shown, via Caccioppoli's inequality, to be controlled by
$|Q| M(N_*(u)^2)(x_*)$;  we omit the routine details).  Covering the region 
$3Q\times (0,2\ell(Q))$ by Whitney boxes (of the decomposition
of the open set $\reu$), and using Caccioppoli's inequality
(as we may, since by $t$-independence, $\partial_t u$ is a solution),
we find that the last term is bounded by a constant times
$$ \int_{4Q} \int_0^{3\ell(Q)} | \partial_t u(y,t)|^2 \,t dtdy \,\lesssim
\,|Q|\,M\left(g(u)^2\right)(x_*)\,.$$
Collecting estimates, we obtain \eqref{eq3.6++}, as claimed.  This concludes Step 1.

To conclude this subsection, let us note that in the context of the Carleson measure estimate of
Corollary \ref{c1}, the preceding argument shows that the left hand side of \eqref{eq1.9} may be replaced by a similar expression, but with $\nabla u$ replaced by $\partial_t u$, modulo errors on the order
of $\|u\|_\infty$.  Thus, to establish Corollary \ref{c1}, it suffices to verify:
\begin{equation*}
\sup_Q \frac1{|Q|} \iint_{T_Q} |\partial_t u(x,t)|^2 t dt dx \leq C\, \|u\|_{L^\infty(\Omega)}\,.
\end{equation*}  
We further note that since $\partial_t u$ is a solution, it satisfies De Giorgi/Nash local H\"older continuity estimates.   Consequently, by \cite[Lemma 2.14]{AHLT}, it is enough to show that there is a
uniform constant $c$, and for each cube $Q$, a set $F\subset Q$, with $|F|\geq c|Q|$,
for which
\begin{equation}\label{eq1.9a}
\frac1{|Q|} \int_F\!\int_0^{\ell(Q)} |\partial_t u(x,t)|^2 t dt dx \leq C\, \|u\|_{L^\infty(\Omega)}\,,
\end{equation} 

\subsection{Step 2:  a ``good-$\lambda$" inequality for the vertical square function}\label{ss3.2}  
We turn now to the heart 
of the proof of Theorem \ref{th1}, namely,
to establish a ``good-$\lambda$" inequality for the vertical square function \eqref{eq3.vert}
in terms of $N_*(u)$.    Throughout this subsection,
we may assume that our solution $u$ is continuous up the the boundary of $\reu$;
indeed, having established
the desired bounds for continuous $u$, we may apply those bounds to
$u_\delta(x,t):= u(x,t+\delta)$, with $\delta>0$ which is a solution
of the same equation, by $t$-independence of the coefficients.   In turn, 
these bounds are preserved in the limit, as $\delta\to 0$, by a monotone convergence argument.
We omit the routine details.

For a given $\lambda >0$, suppose that $Q$ is a Whitney cube for the open set
$$E_\lambda:=\{x\in\rn: M\left(g(u)\right)(x)>\lambda\}\,.$$
We now fix $\eps>0$
so that $2+\eps$
is an exponent for which the Hodge decomposition holds for $L_\|$ and $L_\|^*$
(cf. \eqref{eq1.hodge}-\eqref{eq1.hodge**}.)
Let $\vp,\tilde{\vp}\in W_0^{1,2+\eps}(5Q)$ be as in
\eqref{eq1.hodge}, and for a small $\eta>0$ to be chosen, set
\begin{align}\label{eq3.1*}
\Lambda_1&\,:= \,\eta^{-1}N^\eta_*(\partial_t\pe\vp)+
N_*(\partial_t\p^*_t\vp)+\widetilde{N}^\eta_*(\nabla
\pe\vp) + \left(M(|\nabla_x\vp|^2)\right)^{1/2}
\\[4pt]\label{eq3.2*}
\Lambda_2&\,:=\, \eta^{-1}N^\eta_*(\partial_t\Pe\tilde{\vp})+
N_*(\partial_t\p_t\tilde{\vp})+\widetilde{N}_*^\eta(\nabla
\Pe\tilde{\vp})+ \left(M(|\nabla_x\tilde{\vp}|^2)\right)^{1/2}\,,
\end{align}
where the non-tangential maximal operator $N_*$ in the 
second terms  on the two right hand sides is defined with respect to cones of
aperture 1.  We define a certain ``maximal differentiation operator"
\begin{equation}\label{eq3.MMW}
D_{*,p}f(x):=\sup_{r>0}\left(\fint_{|x-y|<r}\left(\frac{|f(x)-f(y)|}{|x-y|}\right)^p\,dy\right)^{1/p},
\end{equation}
which obeys the estimate
\begin{equation}\label{eq3.MMW2}
\|D_{*,p_1} f\|_p \leq C_{p,p_1,n}\, \|\nabla f\|_p\,,\quad 1\leq p_1<p<\infty\,.
\end{equation}
Indeed, by a classical  ``Morrey type" inequality (see, e.g.,
\cite[Lemma 7.16]{GT}), we have 
$$\frac{|f(x)-f(y)|}{|x-y|} \lesssim M(\nabla f)(x) + M(\nabla f)(y)\,,$$
whence it follows that 
$$D_{*,p_1}f(x)\lesssim M(\nabla f)(x) + \left(M \big(M(\nabla f)\big)^{p_1}(x)\right)^{1/p_1}\,.$$
The latter bound clearly implies \eqref{eq3.MMW2}.

We then fix $p_1\in(1,2)$ and define
\begin{equation}\label{eq3.fdef}
F:=\left\{x\in Q:\Lambda_1(x) + \Lambda_2(x)+D_{*,p_1}\vp(x) +D_{*,p_1}\tilde{\vp}(x)
\leq \kappa_0\right\}\,,
\end{equation}
and note that by \eqref{eq2.19}-\eqref{eq2.21}, \eqref{eq3.MMW2},
and Tchebychev's inequality, we have
\begin{equation}\label{eq3.4*}|Q\setminus F|\lesssim \kappa_0^{-2-\eps}\,|Q|\,,
\end{equation}
uniformly in $\eta$.

Set $p_0:= 2(2+\eps)/\eps$. 
Our goal is to prove that for some aperture $\alpha$ sufficiently large,

\begin{equation}\label{eq3.gl}
\left|\big\{x\in Q: g(u)(x)>3\lambda, 
\left(M\big(N^\alpha_*(u)^{p_0}\big)(x)\right)^{1/p_0}\leq 
\gamma\lambda\big\}\right|\leq C\left(C_{\kappa_0,\eta}\,
\gamma^2\, +\,\kappa_0^{-2-\eps}\right)|Q|\,,
\end{equation}
for all $\gamma >0$ sufficiently small, for all $\kappa_0$ sufficiently large, and for $\eta$ chosen small enough depending on $\kappa_0$.  Here, $\gamma$ is at our disposal, and \eqref{eq3.4*} holds uniformly in $\eta$, so we may choose first $\kappa_0$,
then $\eta$, and finally $\gamma$, to obtain a bound on the RHS of
\eqref{eq3.gl} which is a small portion of $|Q|$, whence the standard good-lambda arguments may be carried out to show that
\begin{equation}\label{eq3.6*} 
\|g(u)\|_p\leq C_p\, \|N_*^\alpha(u)\|_p\,,\qquad \forall p_0<p<\infty\,.
\end{equation}
Let us note at this point that the latter bound, together with \eqref {eq3.global1}, yield that
\begin{equation}\label{eq3.12++} 
\|S(u)\|_p\leq C_p\, \|N_*(u)\|_p\,,\qquad \forall p_0<p<\infty\,.
\end{equation}
By \eqref{eq3.4*}, it is enough to prove the following modified version of \eqref{eq3.gl}:
\begin{equation}\label{eq3.gl2}
\left|\big\{x\in F: g(u)(x)>3\lambda, 
\left(M\big(N^\alpha_*(u)^{p_0}\big)(x)\right)^{1/p_0}\leq 
\gamma\lambda\big\}\right|\,\leq\, C_{\eta,\kappa_0}\,\gamma^2\, |Q|\,,
\end{equation}
As usual, we may assume that there is a point in $Q$, call it $x_*$, such that 
\begin{equation}\label{eq3.x0}
N^\alpha_*(u)(x_*)\leq \left(M\big(N^\alpha_*(u)^{p_0}\big)(x_*)\right)^{1/{p_0}}\leq \gamma\lambda\,,
\end{equation} otherwise there is nothing to prove. 
Let us note that
$$g(u)\,\leq \,\left(\int_0^{\ell(Q)} |\partial_t u|^2\,tdt\right)^{1/2} \,+\,
\left(\int_{\ell(Q)}^\infty |\partial_t u|^2\,tdt\right)^{1/2} =:g_1(u) +g_2(u).$$
We claim that 
\begin{equation}\label{eq3.g2}
g_2(u)(x) \leq (1+C\gamma) \lambda\,,\qquad \forall x\in Q\,.
\end{equation}
Indeed, we have that
$$g_2(u)(x) \leq g(u)(x_Q) + \left(\int_{\ell(Q)}^\infty 
|\partial_t u(x,t)-\partial_t u(x_Q,t)|^2\,tdt\right)^{1/2}\,,$$
where we may choose $x_Q\in \rn\setminus E_\lambda$,
with $\dist(x_Q,Q) \approx \ell(Q)$, since $Q$ is a Whitney cube for $E_\lambda$.
Then $g_2(u)(x_Q)\leq \lambda$, by definition of $E_\lambda$.  Moreover, since our coefficients 
are $t$-independent, we may apply standard De Giorgi/Nash/Moser interior estimates to obtain that
$$\left(\int_{\ell(Q)}^\infty 
|\partial_t u(x,t)-\partial_t u(x_Q,t)|^2\,tdt\right)^{1/2}\lesssim 
\left(\int_{\ell(Q)}^\infty 
\left(\frac{\ell(Q)}{t}\right)^{2\beta}\frac{dt}{t}\right)^{1/2}N^\alpha_*(u)(x_0)\lesssim \gamma\lambda\,,$$
by \eqref{eq3.x0}, where $\beta>0$ is the De Giorgi/Nash exponent, and where
we have taken the aperture $\alpha$ to be sufficiently large.
This proves the claim.  

Taking $\gamma$ sufficiently small, we may therefore suppose that $g_2(u)<2\lambda$
in $Q$, so that the LHS of
\eqref{eq3.gl2} is bounded by
\begin{multline}\label{eq3.tcheb}
\left|\big\{x\in F: g_1(u)(x)>\lambda\big\}\right|\,\leq\,
\frac1{\lambda^2} \int_F \int_0^{\ell(Q)} |\partial_t u|^2 tdtdx\\[4pt] \lesssim
\frac1{\lambda^2} \int_F \int_0^{\ell(Q)} A(x)\nabla u(x,t)\cdot\nabla u(x,t)\,tdt dx\,,
\end{multline}
where in the last step, we have crudely dominated $|\partial_tu|$ by $|\nabla u|$
and then used ellipticity.   We note at this point that in the context of Corollary \ref{c1},
the integral in the middle term is precisely that which appears in \eqref{eq1.9a}.
In the remainder of this subsection, we shall prove that
\begin{equation}\label{eq3.21aa}
\int_F \int_0^{\ell(Q)} A(x)\nabla u(x,t)\cdot\nabla u(x,t)\,tdt dx\,\leq C_{\eta,\kappa_0}\,|Q|
\left(\fint_{2Q}N^\alpha_*(u)^{p_0}\right)^{2/p_0}\,.
\end{equation} 
Clearly, this estimate yields both our desired
``good-lambda" inequality, as well as the bound \eqref{eq1.9a}.

We turn to the proof of \eqref{eq3.21aa}.
By the change of variable $t\to t-\vp(x)+\pe\vp(x)$ (that this change of variable is ``legal"
follows from \eqref{eq3.6} and \eqref{eq3.tder} below),
we have
\begin{equation}\label{eq3.11}
\int_F \int_0^{\ell(Q)} A\nabla u\cdot \nabla u \, tdtdx \lesssim
\int_F \int_0^{2\ell(Q)} A_1\nabla u_1\cdot \nabla u_1 \, tdtdx\,,\end{equation}
where $u_1(x,t):= u(x,t-\vp(x)+\pe\vp(x))$,
and where $A_1$ and $u_1$ are as in Section 2 above.
Here, we have chosen  $\eta\ll\kappa_0^{-2}$, so that
\begin{equation}\label{eq3.6}
|(I-\pe)\vp(x)|=\left|\int_0^{\eta t}\partial_s \p^*_s \vp(x)ds\right|\leq \eta t \kappa_0 \ll\eta^{1/2} t
 \ll t/8\,,\quad \forall x\in F\,.
\end{equation}
We note at this point that the analogue of \eqref{eq3.6} holds
also for $(I-\Pe)\tilde{\vp}$, and moreover, by \eqref{eq3.1*}-\eqref{eq3.fdef},  we have
\begin{equation}
\label{eq3.tder}
\max\left(|\partial_t \Pe\tilde{\vp}(x)|,|\partial_t \pe\vp(x)|\right) \leq \eta \kappa_0 \ll\eta^{1/2}\,,
\qquad \forall (x,t) \in\Omega_0\,,
\end{equation}
where $\Omega_0$ is the sawtooth domain
\begin{equation}
\label{eq3.saw}
\Omega_0:=\bigcup_{x\in F} \,\Gamma_0(x)\,,
\end{equation}
and $\Gamma_0(x)$ denotes the cone with vertex at $x$ 
and aperture $\eta$.   Thus,
if $(x,t)\in \Omega_0$, then $|x-x_0|<\eta t$ for some $x_0\in F$, so that, 
setting $\vp_{x_0,\eta t}
:= \fint_{|x_0-y|<2 \eta t}\vp(y) dy$, we have
\begin{equation}\label{eq3.16a}
|\pe\left(\vp-\vp_{x_0,\eta t}\right)(x)|\,\lesssim \,\eta t M(\nabla\vp)(x_0)\,\lesssim\,\eta t \kappa_0
\ll\eta^{1/2}t\,,\quad \forall (x,t)\in\Omega_0\,,
\end{equation}
by a telescoping argument and Poincar\'e's inequality, and
by the Gaussian bounds for $\pe$.

We now define a smooth cut-off adapted to $\Omega_0$, or to be more precise, to
a slightly smaller sawtooth domain $\Omega_1:= \cup_{x\in F}\Gamma_1(x)$,
where $\Gamma_1(x)$ has aperture $\eta/8$.
Let $\delta(x):= \dist(x,F)$, and let $\Phi\in C^\infty (\re)$, with $\Phi(r) \equiv 1$ if
$r\leq 1/16$, and $\Phi(r)\equiv 0$, if  $r>1/8$.  We then set
\begin{equation}\label{eq3.psidef}
\Psi(x,t)\,:=\,\Phi\left(\frac{\delta(x)}{\eta t}\right)\,\Phi\left(\frac{t}{32\,\ell(Q)}\right)\,.
\end{equation}
Let us record some observations concerning the cut-off $\Psi$, and certain related sawtooth regions.
To begin, we note that 
\begin{equation}\label{eq3.16*}
\Psi(x,t)\equiv 1\,,\qquad \forall (x,t)\in F\times (0,2\,\ell(Q))\,,
\end{equation}
and also, since $\eta$ is small, that 
$$\supp(\Psi)\subset \Omega_{1,Q}:= \Omega_1\cap \Big(2Q\times (0,4\ell(Q))\Big)\,.$$

Next, we claim that
\begin{equation}\label{eq3.18*}
|(I-\pe)\vp(x)| \ll \eta^{1/2} t \,,\qquad \forall (x,t) \in \Omega_{0,Q}:=\Omega_0\cap 
\Big(2Q\times (0,4\ell(Q))\Big) \,,
\end{equation}
and that an analogous bound holds for $(I-\Pe)\tilde{\vp}$.
To verify the claim, 
we first observe that for $(x,t)\in\Omega_0$, there is a point $x_0\in F$ such that
 $$x\in \Delta:= \Delta(x_0,\eta t):=\{x:|x-x_0|<\eta t\}.$$  Let us further observe that  
$2\Delta \subset 5Q$, since 
$t\leq 4\ell(Q)$, and $\eta$ is small.
Next, we note that by \eqref{eq1.hodge}, 
$\vp$ is a $W^{1,2}$ weak solution of the inhomogeneous PDE
$$L^*_\| \vp= \dv ({\bf c})\,,$$
in the domain $5Q$, and the same is true with $\vp$ replaced by $\vp-c$, for any constant $c$.  
Thus, by Moser-type interior estimates, and the definition of $F$ (cf. \eqref{eq3.fdef})
we have that 
\begin{multline}\label{eqmoser*}
\sup_{\Delta}|\vp-\vp(x_0)|\,\lesssim \,
\left(\fint_{2\Delta} |\vp(z)-\vp(x_0)|^{p_1} \,dz\right)^{1/p_1} 
\,+ \,\eta t\,\|{\bf c}\|_\infty\\[4pt]
\lesssim \eta t\,\Big(D_{*,p_1}\vp(x_0) +  \|{\bf c}\|_\infty\Big)\lesssim \eta t\left(\kappa_0
+ \|{\bf c}\|_\infty\right)\ll\eta^{1/2} t\,,
\end{multline}
where the implicit constants depends only upon $p_1$, ellipticity and dimension (see, e.g.,
\cite[ Theorem 8.17, p. 194]{GT}).
Consequently, for {\bf every} $y\in \Delta$, 
we then have
\begin{multline}\label{eq3.18**}
|(I-\pe)\vp(y)|\\[4pt]
\leq\, |\vp(y)-\vp(x_0)| + |(I-\pe)\vp(x_0)|+|\pe\left(\vp-\vp_{x_0,\eta t}\right)(x_0)|+
|\pe\left(\vp-\vp_{x_0,\eta t}\right)(y)| \\[4pt]
\ll \eta^{1/2} t \,,
\end{multline}
where we have used  \eqref{eq3.6} and \eqref{eq3.16a}, along with \eqref{eqmoser*}.  
In particular, since $x\in \Delta$, we obtain \eqref{eq3.18*}, as claimed.
The corresponding bound for $(I-\Pe)\tilde{\vp}$ follows by an identical argument.

Moreover, for $(x,t)\in \Omega_0$, 
by \eqref{eq3.tder} we have
\begin{align}\label{eq3.jbounds}
J(x,t)\,=\,
\partial_t \left(t-\vp(x)+\pe\vp(x)\right)&\, \approx\, 1
\\[4pt]
\label{eq3.jbounds2}
\widetilde{J}(x,t)\,=\,\partial_t \left(t-\tilde{\vp}(x)+\Pe\tilde{\vp}(x)\right)&\, \approx\, 1\,.
\end{align}
We then have that the mapping $\rho(x,t):= (x,\tau(x,t)):= (x,t+\pe\vp(x)-\vp(x))$ is 1-1
on $\supp(\Psi)$, with 
\begin{equation}\label{eq3.23saw}
7t/8< \tau(x,t)<9t/8\,,\qquad \forall (x,t)\in\supp(\Psi)\,.
\end{equation}  
Consequently, if 
$\Omega_\beta:=\cup_{x\in F}\Gamma_\beta(x)$ is the sawtooth domain with respect to $F$, with 
cones of aperture $\beta$, we have that
\begin{equation}\label{eq3.24saw}
\Omega_{8\beta/9}\subset \rho(\Omega_\beta)\subset \Omega_{8\beta/7}\,,
\qquad \forall \beta\leq \eta\,.
\end{equation}

Let us note also that 
\begin{equation}\label{eq3.19*}
|\nabla_{x,t}\Psi(x,t)|\,\lesssim\,
\frac1{\eta t}1_{E_1}(x,t) \,+\,
\frac{1}{\ell(Q)}1_{E_2}(x,t) \,,
\end{equation}
where 
\begin{align}\label{eq3.17}
E_1&\,:=\,\left\{(x,t)\in 2Q\times(0,4\,\ell(Q)): \eta t/16\leq \delta(x) \leq \eta t/8 \right\}
\\[4pt]\nonumber
E_2&\,:=\,2Q\times \Big(2\,\ell(Q),4\,\ell(Q)\Big)
\end{align}

By \eqref{eq3.16*}, we have that the RHS of \eqref{eq3.11} is bounded by
\begin{multline}\label{eq3.12}
\iint_{\mathbb{R}^{n+1}_+}A_1\nabla u_1\cdot\nabla u_1\,\Psi^2\, t\, dt dx =-\frac12
\iint_{\mathbb{R}^{n+1}_+}L_1(u_1^2)\, \Psi^2\,tdt dx\\[4pt]
=-\frac12
\iint_{\mathbb{R}^{n+1}_+}u_1^2\,L^*_1(t)\, \Psi^2 dt dx \,
-\,\frac12\iint_{\reu} A_1\nabla (u_1^2) \cdot\nabla (\Psi^2) t dt dx\\[4pt]
+\,\frac12 \iint_{\reu} (u_1)^2  \,e_{n+1}\cdot A_1\nabla(\Psi^2) \,dx dt\,+\,\frac12
\int_{F} 
u^2\,A_{n+1,n+1}\, dx \\[4pt]
=:\s + \e_1 + \e_2 + \B\,,
\end{multline}
where $e_{n+1}:= (0,...0,1)$, and where in the boundary term $\B$ we have used that 
$(A_1^*)_{n+1,n+1}(x,0) = A_{n+1,n+1}(x)$, that $u_1(x,0) = u(x,0)$ on $F$ (cf. \eqref{eq3.6}),
and that $\Psi(x,0)=1_F(x)$.
We note that
\begin{equation}\label{eq3.22}
|\B|\leq C\,|Q|\fint_Q N^\alpha_*(u)^2 \leq C(\gamma \lambda)^2 |Q|\,,
\end{equation}
by H\"older's inequality and \eqref{eq3.x0}.  Let us 
now consider the ``error terms" $\e_1$ and $\e_2$.  For a small constant $\sigma$ to be chosen later,
we have that
\begin{multline}\label{eq3.23}
|\e_1|\, \leq\, \sigma \iint_{\reu} A_1\nabla u_1\cdot \nabla u_1\, \Psi^2 \,t\, dtdx
\,+\, \frac1{\sigma} \iint_{\reu}u_1^2\,A_1\nabla \Psi\cdot\nabla\Psi \,t\,dtdx\\[4pt]
=: \,\e_1' + \e_1''\,.
\end{multline}
Choosing $\sigma$ small enough, we shall eventually hide $\e_1'$,
along with several
copies of it that will arise later, on the LHS of
\eqref{eq3.12}.   By \eqref{eq3.19*}, and  the 
definition of $A_1$ \eqref{eq2.2}, writing ${\bf h}= {\bf c}1_{5Q}
+A_\|^*\nabla \vp$ (cf. \eqref{eq1.hodge}), and using \eqref{eq3.jbounds}
and the fact that the original coefficient matrix is bounded,
we find that
$$\e''_1\leq \e''_{11}+\e''_{12}\,,$$
where
$$\e''_{11}=\frac{C_\eta}{\sigma} \iint_{E_1} 
u_1^2\,\Big[1+ |\nabla_x(I-\pe)\vp(x)|^2\Big]\,\frac{dx dt}{t}\,,$$
and where $\e''_{12}$ is a similar integral over the region
$E_2$.  We shall treat only $\e''_{11}$, as the 
term $\e''_{12}$ is easier. 

To this end, we write
\begin{equation}\label{eq3.28*}
\e_{11}'' = \frac{C_\eta}{\sigma}\sum_k\sum_{Q'\in \dd_k^\eta}\int_{Q'}\int_{2^{-k}}^{2^{-k+1}}
u_1^2\,\Big(1+ |\nabla_x(I-\pe)\vp(x)|^2\Big)\,1_{E_1}\,\frac{dx dt}{t}\,,
\end{equation}
where $\dd_k^\eta$ denotes the grid of dyadic cubes such that
\begin{equation}\label{eq3.dyadicgrid}
\frac1{64} \eta2^{-k}\leq\diam Q'< \frac1{32} \eta 2^{-k}\,,\qquad Q'\in \dd_k^\eta\,.
\end{equation}
Consider now any fixed $k$ and $Q'\in\dd_k^\eta$, for which the double integral
in \eqref{eq3.28*} is non-zero, thus, 
for which there is a point 
\begin{equation}\label{eqintersect}
(x_1,t_1)\in E_1\cap \left(Q'\times \left[2^{-k},2^{-k+1}\right]\right).
\end{equation}
We now fix such a point $(x_1,t_1)$.  By definition of $E_1$,
\begin{equation}\label{eq3.35****}
\frac{\eta t_1}{16}\leq \delta(x_1)\leq \frac{\eta t_1}8\,.\end{equation}
In particular, there is a point $x_0\in F$ such that $|x_1-x_0|<(\eta t_1)/8$.  
Note that 
\begin{equation}\label{eqndisk}
Q'\subset \Delta':=\Delta(x_0,\eta 2^{-k}):= 
\{z: |x_0-z|<\eta 2^{-k}\}\,,
\end{equation} 
by \eqref{eq3.dyadicgrid}.  Consequently, 
\begin{equation}\label{eq3.34***} 
Q'\times[2^{-k},2^{-k+1}]\subset \Omega_{0,Q}
\end{equation}
(we recall that $\Omega_{0,Q}$ is defined in \eqref{eq3.18*}).
Furthermore, since $\delta$ is Lipschitz with norm 1, using \eqref{eq3.dyadicgrid}
and \eqref{eq3.35****}, we obtain that
there is a uniform constant $C$ such that
\begin{equation}\label{eq3.36***}
Q'\times\left[2^{-k},2^{-k+1}\right]\,\subset\, \widetilde{E}_1:=\Big\{(y,s)
\in 2Q\times(0,4\ell(Q)): \frac{\eta s}C\leq\delta(y)\leq C\eta s\Big\} \,.
\end{equation}
It then follows that 
\begin{equation}\label{eq3.38***}
|Q'|\lesssim \int_{Q'}\int_{2^{-k}}^{2^{-k+1}}1_{\widetilde{E}_1}(y,s) \frac{ds}{s}dy\,.
\end{equation}

Now, by \eqref{eq2.22*}, \eqref{eq3.1*}, and \eqref{eq3.fdef}, 
we have that for every $t\in[2^{-k},2^{-k+1}]$, 
\begin{multline}\label{eq3.29a}
\fint_{Q'}|\nabla_x(I-\pe)\vp(x)|^2\,dx\,\lesssim\,\fint_{\Delta'}|\nabla_x\pe\vp(x)|^2\,dx \,+\,
\fint_{\Delta'}|\nabla_x\vp(x)|^2\,dx
 \\[4pt]\lesssim  \,\left(\widetilde{N}^\eta_*(\nabla \pe\vp)\right)^2(x_0) +
M(|\nabla_x\vp|^2)(x_0)\,\lesssim\, \kappa_0^2\,.
\end{multline}
Moreover, by \eqref{eq3.18*}, \eqref{eq3.34***}, and the definition of $u_1$, for $\alpha$ large
enough we have
\begin{equation}\label{eq3.33b}
\sup |u_1(x,t)|\leq \essinf_{y\in Q'} N_*^\alpha(u)(y)\,,
\end{equation}
where the supremum runs over all $(x,t)\in Q'\times (2^{-k},2^{-k+1})$.
Thus,
\begin{multline}\label{eq3.29*}
\int_{Q'}\int_{2^{-k}}^{2^{-k+1}}
u_1^2\,\Big(1+ |\nabla_x(I-\pe)\vp(x)|^2\Big)\,1_{E_1}\,\frac{dx dt}{t}\\[4pt]
\leq\int_{2^{-k}}^{2^{-k+1}} \!\!
\essinf_{Q'} \left(N_*^\alpha (u)^2\right)\fint_{Q'}\Big(1+ |\nabla_x(I-\pe)\vp(x)|^2\Big) \,dx\,\,
|Q'| \,\frac{ dt}{t}\\[4pt]
\lesssim\, \left(1+\kappa_0^2\right) 
\int_{Q'}N_*^\alpha (u)^2(y) \int_{2^{-k}}^{2^{-k+1}}1_{\widetilde{E}_1}(y,s) \, \frac{ds}{s}dy\,,
\end{multline}
where we have used \eqref{eq3.38***} and \eqref{eq3.29a}.
Returning to \eqref{eq3.28*}, we then have
\begin{multline*}
\e_{11}'' \leq C_{\eta,\kappa_0, \sigma}\sum_k\sum_{Q'\in \dd_k^\eta}\int_{Q'}N_*^\alpha (u)^2(y)
 \int_{2^{-k}}^{2^{-k+1}}1_{\widetilde{E}_1}(y,s) \, \frac{ds}{s}\, dy\\[4pt]
 \leq C_{\eta,\kappa_0, \sigma} \int_{2Q}N_*^\alpha (u)^2(y)
  \int_{\delta(y)/(C\eta)}^{C\delta(y)/\eta} \, 
  \frac{ds}{s}\,dy \leq C_{\eta,\kappa_0, \sigma} (\gamma\lambda)^2  |Q|\,,
 \end{multline*}
 where in the last step we have used \eqref{eq3.x0}.

The term $\e_2$ in \eqref{eq3.12} satisfies the same bounds as $\e''_1$.
It therefore remains to treat the main term $\s$.    
To this end, we first observe that
$$
L_1^*(t) = \dv_x A_\|^*\nabla_x \pe \vp 
\,-\,\partial_t \left(\frac1J\langle A \,{\bf p},{\bf p}\rangle\right)\,=:
-L^*_\|\pe\vp \,-\,\partial_t \left(\frac1J\langle A \,{\bf p},{\bf p}\rangle\right)\,,
$$
since $\dv_x {\bf h} = 0$.
We then have that
\begin{multline}\label{eq3.31sdef}
\s=\frac12
\iint_{\mathbb{R}^{n+1}_+}u_1^2\,\left(L^*_\| \pe \vp\right) \, \Psi^2 dt dx\,+\,
\frac12
\iint_{\mathbb{R}^{n+1}_+}u_1^2\,\partial_t \left(\frac1J\langle A \,{\bf p},{\bf p}\rangle\right) 
\, \Psi^2 dt dx\\[4pt]=:\, \s_1+\s_2\,.
\end{multline}
We treat $\s_1$ first.  We note that by definition of $\pe$, we have 
\begin{equation}\label{eq3.31a}
 \partial_t\pe=-2\eta^2tL^*_\|\pe=-2\eta^2t\,\pe \,L^*_\|\,.
\end{equation} Integrating by parts in $t$, we then obtain 
\begin{multline}
\s_1 \, =\,-\frac12 \iint_{\mathbb{R}^{n+1}_+}u_1^2\, \partial_t
\left(L^*_\| \pe \vp\right)\,\Psi^2\, t \,dt dx\\[4pt]+\,C_\eta
\iint_{\mathbb{R}^{n+1}_+}\left(u_1\,\partial_t u_1\right)\,
\partial_t\pe\vp\,\Psi^2\, dt dx \, +\,C_\eta
\iint_{\mathbb{R}^{n+1}_+}\,u^2_1\,
\partial_t\pe\vp\,\left(\Psi\, \partial_t \Psi\right)\,dt dx\\[4pt]=: \s_1' + \s_1''+\s_1'''.
\label{eq2.4}\end{multline}  
The term $\s_1'''$ may be handled like $\e_1''$ and $\e_2$ above,
except that the present term is somewhat easier, since $\partial_t\pe\vp$ is bounded
in the support of $\Psi$ (cf. \eqref{eq3.1*} and \eqref{eq3.fdef}.)

Next, using  \eqref{eq3.31a}, and that the original matrix $A\in L^\infty$, we have
\begin{multline}|\s_1'| \leq \left|\iint_{\mathbb{R}^{n+1}_+}u_1 \,\nabla_x u_1\cdot A^*_\| 
\nabla_x \partial_t
\pe\vp\,\Psi^2\, t \,dt dx\right|\\[4pt] +\,
\left|\iint_{\mathbb{R}^{n+1}_+}u_1^2\, \left(A^*_\| 
\nabla_x \partial_t
\pe\vp\cdot \nabla_x \Psi\right)\,\Psi\, t \,dt dx\right|\,=:\, J + K
\\[4pt] \lesssim \sigma
\iint_{\mathbb{R}^{n+1}_+}|\nabla_x u_1|^2 \,\Psi^2\,t\,dtdx\,
+\, \left(\frac1{\sigma}+ 1\right)
\iint_{\mathbb{R}^{n+1}_+}u_1^2\,\left|\eta^2\nabla_x 
\pe\,L_\|^*\vp\right|^2\,\Psi^2\,
t^3dtdx\\[4pt] +\,
\iint_{\mathbb{R}^{n+1}_+}u_1^2\,  
|\nabla_x \Psi|^2 \,t \,dt dx := \s_{11}' + \s_{12}'\, +\s_{13}',
\label{eq2.5}\end{multline}
where once again $\sigma$ is a small number at our disposal.
The term $\s_{13}'$ is a slightly simpler version of $\e_1''$, and may be handled by a similar
argument.

Next, we consider $\s_{12}'$.   By \eqref{eq3.18*}, and the definition of $u_1$,
we have that
\begin{equation}\label{eq3.34}
|u_1(x,t)|\leq \sup_{s>0}|u(x,s)|\leq N_*^\alpha(u)(x)\,,\qquad \forall (x,t)\in\Omega_{0,Q}\,.
\end{equation}
Consequently,
\begin{multline}\label{eq3.40s12}
\s_{12}'\leq C_{\sigma} \int_{2Q}N_*^\alpha(u)^2(x)\,\left(\widetilde{\G}_2(A^*_\|\nabla\vp)(x)\right)^2\,dx\\[4pt]
\leq C_\sigma \left(\int_{2Q}N_*^\alpha(u)^{2(2+\eps)/\eps}\,dx\right)^{\eps/(2+\eps)}
 \left(\int_{\rn}\left(\widetilde{\G}_2(A^*_\|\nabla\vp)\right)^{2+\eps}\,dx\right)^{2/(2+\eps)}\\[4pt]
 \leq C_\sigma (\gamma\lambda)^2 |Q|\,,
\end{multline}
where $\widetilde{G}_2$ is the $\p^*_t$ analogue of the vertical square function defined
in \eqref{eq1.gdef2}, and where we have used \eqref{eq1.katovertp}, \eqref{eq1.hodge*},
and \eqref{eq3.x0}
(with $p_0:= 2(2+\eps)/\eps$.)

We would like to handle $\s_{11}'$ by simply hiding it on the LHS of \eqref{eq3.12},
with $\sigma$ chosen small enough, but there is a slightly delicate issue of ellipticity that 
we must address in order to do this.  Before doing so, let us observe that 
$$|\s_1''|\leq \sigma\! \iint_{\reu} |\partial_t u_1|^2 \, \Psi^2\, t dt dx \,+\,C_{\eta,\sigma}\!\iint_{\reu}
|u_1|^2\,|\partial_t\pe\vp|^2\, \Psi^2 \,\frac{dxdt}{t}=: \s_{11}''+\s_{12}''\,.$$
The term $\s_{12}''$ may be handled exactly like $\s_{12}'$ above, but with the 
$\p^*_t$ analogue of \eqref{eq1.gdef1} in place of \eqref{eq1.gdef2}, and we obtain the bound
$$\s_{12}'' \leq C_{\eta,\sigma} (\gamma\lambda)^2|Q|\,.$$
The term $\s_{11}''$ is of the same nature as $\s_{11}'$, and we shall treat them together.
In fact,
\begin{equation}\label{eq3.35}
\s_{11}'+\s_{11}'' = \sigma \iint_{\reu} |\nabla u_1|^2\,\Psi^2\,t dt dx\,,
\end{equation}
where, unless otherwise specified, $\nabla:= \nabla_{x,t}$.	
We recall that $u_1=u\circ \rho$, with 
$$\rho(x,t):= (x,t+\pe\vp(x)-\vp(x))=:(x,\tau(x,t))\,.$$
Thus, 
\begin{align}\label{eq3.36}
\partial_t u_1 (x,t)&= J(x,t)  (\partial_\tau u)(x,\tau(x,t))\\[4pt]\label{eq3.37}
\nabla_x u_1(x,t)&= \Big(\nabla_x u\Big)(x,\tau(x,t)) + 
 \Big(\partial_\tau u\Big)(x,\tau(x,t)\Big(\nabla_x\tau(x,t)\Big)\,,
\end{align}
where $J(x,t):= \partial_t \tau(x,t) = 1 +\partial_t\pe\vp(x)$. Consequently,
$$(\nabla u)\circ \rho = 
\left(\nabla_xu_1- \frac{\partial_t u_1}{J}\Big(\nabla_x\tau\Big), \frac{\partial_t u_1}{J}\right)$$
Since $J\approx 1$ in $\Omega_0$, we have that
\begin{multline*}|\nabla u_1|\, \lesssim \,\left|\left(\nabla_x u_1, \frac{ \partial_t u_1}{J}\right)\right|\\[4pt]
\lesssim\, \left|\left(\nabla_x u_1- \frac{\partial_t u_1}{J}\Big(\nabla_x\tau\Big), \frac{ \partial_t u_1}{J}\right)\right|\,+\, |\nabla_x\tau|\,|\partial_t u_1|\, =\,|(\nabla u)\circ \rho| \,+\,|\nabla_x\tau|\,|\partial_t u_1|\,.
\end{multline*}
By \eqref{eq2.10}, the ellipticity of $A$, and the fact that $J\approx 1$, we have
that
\begin{equation}\label{eq3.40*}
|(\nabla u)\circ\rho|^2\, \lesssim\, A_1\nabla u_1\cdot \nabla u_1\,.
\end{equation} 
The latter term gives a contribution to \eqref{eq3.35}
that may be hidden on the LHS of \eqref{eq3.12}, if $\sigma$ is chosen small enough.
It remains to treat $|\nabla_x\tau|\,|\partial_t u_1|$.  To this end, we make the same 
dyadic decomposition as in \eqref{eq3.28*}-\eqref{eq3.dyadicgrid} to write

\begin{equation}\label{eq3.39}
\iint_{\reu} |\nabla_x\tau|^2\,|\partial_t u_1|^2\,\Psi^2\, t \,dxdt
= \sum_k\sum_{Q'\in \dd_k^\eta}\int_{2^{-k}}^{2^{-k+1}}\!\!\!\int_{Q'}
|\nabla_x\tau|^2\,|\partial_t u_1|^2\,\Psi^2\, t \,dx dt\,.
\end{equation}
Consider now some $t_1\in[2^{-k},2^{-k+1}]$ and a cube $Q'\in \dd_k^\eta$ for which
$Q'\times\{t_1\}$ meets $\supp(\Psi)$, say at the point $(x_1,t_1)$.  Then $\delta(x_1)<\eta t_1/8$,
by the construction of $\Psi$, whence by \eqref{eq3.dyadicgrid}, we have
$\delta(x) <\eta t_1/4$, for every $x\in Q'$.  Thus, for each $Q'$ and $t_1$ as above,
there is a point $x_0\in F$ and an $n$-disk $\Delta'$ such that \eqref{eqndisk}, and thus also
\eqref{eq3.34***} and  \eqref{eq3.29a}, hold.  In particular,
\begin{equation*}
\frac{7}8t < \tau(x,t) <\frac{9}{8}t\,,\qquad \forall (x,t) \in I(Q'):=Q'\times[2^{-k},2^{-k+1}]\,,
\end{equation*}
by \eqref{eq3.18*} and the definition of $\tau(x,t)$.  It then follows that for $t\in[2^{-k},2^{-k+1}]$,
$$\sup_{x\in Q'}|\partial_t u_1(x,t)|\approx
\sup_{x\in Q'}|(\partial_\tau u)(x,\tau(x,t))|\lesssim \left((\eta t)^{-n-1}\int_{2Q'}\int_{t/2}^{2t}
|\partial_s u(y,s)|^2 ds dy\right)^{1/2}\,,$$
by \eqref{eq3.jbounds}, \eqref{eq3.36},
Moser's interior estimates, and the $t$-independence of $A$.  

We let $\dd_k^\eta(\Psi)$ denote those $Q' \in \dd_k^\eta$ for which
$I(Q'):=Q'\times[2^{-k},2^{-k+1}]$ meets $\supp(\Psi)$;  thus, for which
there is a point $(x,t)\in I(Q')$ such that $\delta(x)<\eta t/8$, by construction of
$\Psi$.  Consequently, for any such $Q'$, by \eqref{eq3.dyadicgrid} we have that 
$$\delta(y) <\diam(2Q') + \frac18 \eta t \leq \frac3{16} \eta t \leq \frac38\eta\, s
\,,\qquad \forall y\in 2Q',\, s>t/2\,.$$
Moreover, we have $t< 4\ell(Q)$ in $\supp(\Psi)$, so that  $s\leq 2t$ implies $s <8\,\ell(Q)$.
Set $\Omega^*:= \{(y,s)\in \reu: \delta(y)< 3\eta s/8, 0<s<8\,\ell(Q)\}$.   As noted above,
 \eqref{eq3.29a} holds in the present context, so that
 \eqref{eq3.39}
is bounded by a constant times
\begin{multline}\label{eq3.42}
\frac1\eta\sum_k\sum_{Q'\in \dd_k^\eta}\int_{2^{-k}}^{2^{-k+1}}\!\!\!\fint_{Q'}
|\nabla_x\tau|^2\,dx\int_{t/2}^{2t}\!\int_{2Q'}
|\partial_s u(y,s)|^2 1_{\Omega^*}(y,s)\,dyds\,   dt\\[4pt]
\leq C_{\eta,\kappa_0}
\sum_k\sum_{Q'\in \dd_k^\eta}\int_{2^{-k-1}}^{2^{-k+2}}\!\!
\int_{2Q'} |\partial_s u(y,s)|^2 1_{\Omega^*}(y,s)\,s\,dyds\\[4pt]
=C_{\eta,\kappa_0}\left(\iint_{\Omega^{**}} |\partial_s u|^2 \,s\,dyds\,+\, \iint_{\Omega^*\setminus\Omega^{**}}
|\partial_s u|^2 \,s\,dyds\right)\,=:\m+\e\,,
\end{multline}
where
$$\Omega^{**}:=\{(y,s)\in \reu: \delta(y)< \eta s/18, 0<s<\,\ell(Q)\}\,.$$

We observe that by
\eqref{eq3.23saw}-\eqref{eq3.24saw}, we have 
$$\rho^{-1}(\Omega^{**})\subset 
\Omega_{\eta/16}\cap (2Q\times (0,2\ell(Q))\,,$$
and we note that $\Psi\equiv 1$ on the latter set.  Therefore, making the change of variable
$s=\tau(y,t)$, we find that 
$$\m\leq C_{\eta,\kappa_0}
\iint_{\reu} |(\partial_\tau u)\circ \rho|^2 \,\Psi^2\, t\, dt dy\,,$$
since, as above, $J(y,t) \approx 1$.  By \eqref{eq3.40*},
the latter term gives a contribution to \eqref{eq3.35}
that may be hidden on the LHS of \eqref{eq3.12}, if $\sigma$ is chosen small enough.

To handle the error term $\e$, we first note that by Moser's interior estimates, and the $t$-independence of $A$, we have
$$|\partial_s u(y,s)|\lesssim \frac1s\,N_*^\alpha (u)(y)\,.$$
Thus, by definition of $\Omega^*\setminus\Omega^{**}$,  we have
$$\e \leq C_{\eta,\kappa_0}\int_{2Q}\Big(N_*^\alpha (u)(y)\Big)^2 
\left(\int^{18\delta(y)/\eta}_{8\delta(y)/(3\eta)} \frac{ds}{s} +\int_{\ell(Q)}^{8\ell(Q)} \frac{ds}{s}\right)\,dy
\,\leq C_{\eta,\kappa_0} (\gamma \lambda)^2 |Q|\,,$$
where in the last step we have used \eqref{eq3.x0}
and H\"older's inequality.  This concludes our treatment of the term $\s_1$
in \eqref{eq3.31sdef}.  It remains only to treat the term $\s_2$.

To this end, we write
\begin{multline}\label{eq3.45} 2\,\s_2=
\iint_{\mathbb{R}^{n+1}_+}u_1^2\,\partial_t 
\left(\frac1J\langle A \,{\bf p},{\bf p}\rangle\right)\, \Psi^2 dt dx
=\iint_{\mathbb{R}^{n+1}_+}u_1^2\,\partial_t 
\left(\frac1J\right)\langle A \,{\bf p},{\bf p}\rangle\, \Psi^2 dt dx\\[4pt]
+\,\iint_{\mathbb{R}^{n+1}_+}u_1^2\,
\frac1J\langle  \,\partial_t {\bf p},A^*{\bf p}\rangle\, \Psi^2 dt dx
\,+\,\iint_{\mathbb{R}^{n+1}_+}u_1^2\,
\frac1J\langle  A\,{\bf p},\partial_t {\bf p}\rangle\, \Psi^2 dt dx\\[4pt]
=:I+II+III\,,
\end{multline}
where we have used that $A$ is $t$-independent.

We treat these terms in order.  We recall that $J(x,t) = 1 + \partial_t \pe\vp(x)$.  
Then 
\begin{multline*}
I = -\iint_{\mathbb{R}^{n+1}_+}u_1^2\,
\frac{\partial_t^2 \pe\vp}{J^2}\,\langle  A\,{\bf p},{\bf p}\rangle\, \Psi^2 \,dt dx\\[4pt]
= \iint_{\mathbb{R}^{n+1}_+}\partial_t\left(u_1^2\right)\,
\frac{\partial_t \pe\vp}{J^2}\,\langle  A\,{\bf p},{\bf p}\rangle\, \Psi^2\,dt dx
\,+\, \iint_{\mathbb{R}^{n+1}_+}u_1^2\,
\frac{\partial_t \pe\vp}{J^2}\,\partial_t\langle  A\,{\bf p},{\bf p}\rangle\, \Psi^2 \,dt dx\\[4pt]
+\,\iint_{\mathbb{R}^{n+1}_+}u_1^2\,
\partial_t \pe\vp\,
\partial_t\left(\frac1{J^2}\right)\,\langle  A\,{\bf p},{\bf p}\rangle\, \Psi^2 \,dt dx\\[4pt]
+\, \iint_{\mathbb{R}^{n+1}_+}u_1^2\,
\frac{\partial_t \pe\vp}{J^2}\,\langle  A\,{\bf p},{\bf p}\rangle\, \partial_t\left(\Psi^2\right) \,dt dx
\,=:\, I_1+I_2+I_3+I_4\,,
\end{multline*}
where we have used that the boundary terms vanish, since $\partial_t\pe\vp\big|_{t=0} = 0$
(as may be seen by first considering $\vp$ in the domain of $L_\|^*:=-\dv A_\|^*\nabla$, and then using a density argument).    

We recall that ${\bf p}:= (\nabla_x(\pe-I)\vp(x),-1)=(\nabla_x\tau(x,t),-1)$. 
Since $\partial_t \pe\vp$ is bounded, and $J\approx 1$, in $\Omega_0$,
the term $I_4$ may then be handled exactly like the terms $\e_{11}''$ and $\e_{12}''$.

The other terms will require some further work.  To begin,
\begin{equation}\label{eq3.47}
|I_1|\leq\, \sigma\iint_{\reu}|\partial_t u_1|^2\,|{\bf p}|^2\,\Psi^2\, t \,dt dx\,+\,
\frac{C}{\sigma} \iint_{\reu}u_1^2\,|\partial_t \pe\vp|^2\,|{\bf p}|^2\,\Psi^2\,  \frac{dt dx}{t}\,.
\end{equation}
By definition of {\bf p}, the first of these terms may be handled exactly like
\eqref{eq3.39}, and hidden on the LHS of \eqref{eq3.12}, if $\sigma$ is chosen small enough.
The second term is treated via the same dyadic decomposition as above:
\begin{equation*}
\iint_{\reu}u_1^2\,|\partial_t \pe\vp|^2\,|{\bf p}|^2\,\Psi^2\,  \frac{dt dx}{t}
\,=\,\sum_k\sum_{Q'\in \dd_k^\eta}\int_{Q'}\int_{2{-k}}^{2^{-k+1}}
u_1^2\,|\partial_t \pe\vp|^2\,|{\bf p}|^2\,\Psi^2\,  \frac{dt dx}{t}\,,
\end{equation*}
and in turn we note that
\begin{multline*}
\int_{Q'}\int_{2{-k}}^{2^{-k+1}}
u_1^2\,|\partial_t \pe\vp|^2\,|{\bf p}|^2\,\Psi^2\,  \frac{dt dx}{t}\\[4pt]
\lesssim\,\int_{2{-k}}^{2^{-k+1}}\!
\Big(\essinf_{Q'}(N_*^\alpha(u))\Big)^2
\,\left(\int_{2Q'}\!\int_{2^{-k-1}}^{2^{-k+2}}|\partial_s \p_{\eta s}^*\vp(y)|^2\,\frac{dy ds}{s}\right)
\fint_{Q'}|{\bf p}|^2\,\Psi^2\,  \frac{dxdt}{t}\\[4pt]
\lesssim\,C_{\kappa_0}
\,\left(\int_{2Q'}(N_*^\alpha(u))^2\,\int_{2^{-k-1}}^{2^{-k+2}}|\partial_s \p_{\eta s}^*\vp(y)|^2\,1_{\Omega_0}\frac{dy ds}{s}
\right)\,,
\end{multline*}
where we have used \eqref{eq3.33b}, and Moser's parabolic local interior estimates
(of course, accounting for the rescaling $t\to t^2$) in the first inequality, and
\eqref{eq3.29a} (which holds in the present situation), along with the definitions of 
$\Psi$ and $\Omega_0$ in the second.  At this point, we may sum in $Q'$ and in $k$,
and then argue as in our treatment of $\s_{12}'$ above (cf. \eqref{eq3.40s12}),  
using \eqref{eq1.katovertp}
(or rather its analogue for $\p_t^*$),
to obtain a bound on the order of $C_{\sigma,\kappa_0}(\gamma\lambda)^2|Q|$, as desired.

Next, we consider the term $I_2$, which by definition of {\bf p} satisfies the bound
\begin{equation}\label{eq3.49}
|I_2|\,\lesssim \,\iint_{\reu}u_1^2\,|\nabla_x\partial_t \pe\vp|^2\,\Psi^2\,  t\, dt dx \,+\,
\iint_{\reu}u_1^2\,|\partial_t \pe\vp|^2\,|{\bf p}|^2\,\Psi^2\,  \frac{dt dx}{t}\,.
\end{equation}  
But the terms above are both OK, since the first is the same as $\s_{12}'$ in \eqref{eq2.5},
and the second is the same as the second term on the RHS of \eqref{eq3.47}.  We therefore obtain the bound $|I_2|\lesssim (\gamma\lambda)^2|Q|$.

To conclude our treatment of term $I$, we observe that by definition of $J$, we have
\begin{equation*}
|I_3|\,\lesssim\,\iint_{\reu}u_1^2\,|\partial_t^2\, \pe\vp|^2\,\Psi^2\,  t\, dt dx \,+\,
\iint_{\reu}u_1^2\,|\partial_t \pe\vp|^2\,|{\bf p}|^2\,\Psi^2\,  \frac{dt dx}{t}\,.
\end{equation*}
Except for the $t$-derivative in place of $\nabla_x$ in the first term, this is exactly the same bound
as we had for $I_2$, and these terms may therefore be handled in exactly the same way.

Next we treat term $II$.  By definition of {\bf p}, we have $\partial_t{\bf p} = (\nabla_x\partial_t\pe\vp,0)$,
whence it follows from \eqref{eq1.hodge} that, for $x\in 5Q$,
\begin{align}\nonumber
\langle \partial_t {\bf p},A^*{\bf p}\rangle &= \langle \nabla_x\partial_t\pe\vp,A_\|^*
\nabla_x\pe\vp\rangle \,-\, \langle \nabla_x\partial_t\pe\vp,A_\|^*
\nabla_x\vp\rangle - \langle \nabla_x\partial_t\pe\vp,
{\bf c}\rangle \\[4pt]\label{eq3.50a}
&= \langle \nabla_x\partial_t\pe\vp,A_\|^*
\nabla_x\pe\vp\rangle \,-\, \langle \nabla_x\partial_t\pe\vp,{\bf h}\rangle
\end{align}
Thus,
\begin{multline*}
II =\\[4pt] \iint_{\mathbb{R}^{n+1}_+}u_1^2\,\frac1J
\,\langle  \,\nabla_x\partial_t\pe\vp,A_\|^*\nabla_x\pe\vp\rangle\, \Psi^2 dt dx\,\,
-\,\iint_{\mathbb{R}^{n+1}_+}u_1^2\,\frac1J
\,\langle  \,\nabla_x\partial_t\pe\vp,{\bf h}\rangle\, \Psi^2 dt dx\\[4pt]
=:II_1 + II_2\,.
\end{multline*}
In turn,
\begin{multline*}
II_1 \,=\,\iint_{\mathbb{R}^{n+1}_+}u_1^2\,\frac1J
  \,\left(\partial_t\pe\vp\right)\,\left(L_\|^*\pe\vp\right)\, \Psi^2 dt dx\\[4pt] 
-\,\iint_{\mathbb{R}^{n+1}_+}\partial_t\pe\vp\,\langle\,\nabla_x\left(u_1^2\,\frac1J\right)
 \,,\,A_\|^*\nabla_x\pe\vp\rangle\, \Psi^2 dt dx \\[4pt]
 -\,\iint_{\mathbb{R}^{n+1}_+}u_1^2\,\frac1J\,\partial_t\pe\vp\,\langle\,\nabla_x\left( \Psi^2\right)
 \,,\,A_\|^*\nabla_x\pe\vp\rangle\, dt dx\\[4pt]
=:II_1'+II_1''+II_1'''\,.
\end{multline*}
Since $L_\|^*\pe = -(2\eta^2 t)^{-1} \partial_t\pe$, 
the term $II_1'$ is like the second term
on the RHS of \eqref{eq3.47}, only a bit simpler, as we just have 1 in place of {\bf p}. 

Distributing $\nabla_x$, and using that $J\approx 1$, and that $\nabla_x J= \nabla_x\partial_t
\pe\vp$, we have that
\begin{multline}\label{eq3.50}
|II_1''| \leq \, \sigma \iint_{\mathbb{R}^{n+1}_+}|\nabla_x u_1|^2\, \Psi^2 \,t\,dt dx \,+\,C
\iint_{\mathbb{R}^{n+1}_+} u_1^2
 \,|\nabla_x\partial_t\pe\vp|^2\, \Psi^2 \,t\,dt dx \\[4pt]+\,C\left(\sigma^{-1}+1\right)
\iint_{\mathbb{R}^{n+1}_+}u_1^2\,|\partial_t\pe\vp|^2\,
 \,|\nabla_x\pe\vp|^2\, \Psi^2 \frac{dt dx}{t} \,.
\end{multline}
The first of these terms is bounded by \eqref{eq3.35}, and may 
therefore be treated in exactly the same way.
The second and third terms are essentially like the two terms bounding
$I_2$ in \eqref{eq3.49}, since in the last term
we may handle the factor
$|\nabla_x\pe\vp|^2$ just like $|{\bf p}|^2$, using
\eqref{eq3.29a}.

To complete our treatment of $II_1$, we observe that
$$|II_1'''|\lesssim \,\iint_{\mathbb{R}^{n+1}_+}u_1^2
\, |\nabla_x \Psi|^2\,t \,dt dx\,
 +\,\iint_{\mathbb{R}^{n+1}_+}u_1^2\,|\partial_t\pe\vp|^2\,
 \,|\nabla_x\pe\vp|^2\, \Psi^2 \frac{dt dx}{t}\,.$$
 The first of these is the same as $\s_{13}'$ in \eqref{eq2.5}, and the second is the same as the last term
 in \eqref{eq3.50}.

Next, we consider $II_2$.   Since {\bf h} is divergence free,
$$II_2 = \iint_{\mathbb{R}^{n+1}_+}\partial_t\pe\vp\,\langle\,
\nabla_x\left(\frac{u_1^2}J\right),{\bf h}\rangle\, \Psi^2 dt dx \,+\,
 \iint_{\mathbb{R}^{n+1}_+}\frac{u_1^2}J\,\partial_t\pe\vp\,
 \langle\,\nabla_x\left( \Psi^2\right),{\bf h}\rangle\, dt dx\,. $$
 The first of these terms may be treated exactly like $II_1''$ above, and the second exactly
 like $II_1'''$, since ${\bf h} = {\bf c}1_{5Q}
 +A_\|^*\nabla_x\vp$, and therefore may be handled via
 \eqref{eq3.29a}, just like the factor $\nabla_x\pe\vp$.

Last, we consider term $III$.  By an identity analogous to \eqref{eq3.50a}, we have
\begin{multline*}
III \,=\, \iint_{\mathbb{R}^{n+1}_+}u_1^2\,\frac1J
\,\langle A_\| \nabla_x\pe\vp\,,\nabla_x\partial_t\pe\vp\rangle\, \Psi^2 dt dx\\[4pt]
-\,\,\iint_{\mathbb{R}^{n+1}_+}u_1^2\,\frac1J
\,\langle\,{\bf b} +A_\|\nabla_x\vp \, ,\nabla_x\partial_t\pe\vp\rangle\, \Psi^2 dt dx\\[4pt]
=\, \iint_{\mathbb{R}^{n+1}_+}u_1^2\,\frac1J
\,\langle  \,\nabla_x\left(\pe\vp-\vp\right),A_\|^*\nabla_x\partial_t\pe\vp\rangle\, \Psi^2 dt dx\\[4pt]
-\,\,\iint_{\mathbb{R}^{n+1}_+}u_1^2\,\frac1J
\,\langle\,{\bf b}   \,,\,\nabla_x\partial_t\pe\vp\rangle\, \Psi^2 dt dx\,\,
=:\,III_1 + III_2\,.
\end{multline*}
In turn,
\begin{multline*}
III_1 \,=\, -
 \iint_{\mathbb{R}^{n+1}_+}\left(\pe\vp-\vp\right)\,\langle  \,\nabla_x\left(u_1^2\,\frac1J\,\Psi^2\right)
\,,A_\|^*\nabla_x\partial_t\pe\vp\rangle\,  dt dx\\[4pt]
-\,\,\iint_{\mathbb{R}^{n+1}_+}u_1^2\,\frac1J\,\left(\pe\vp-\vp\right)
\,\left(L_\|^*\partial_t\pe\vp\right)\, \Psi^2\, dt dx\,=:III_1' + III_1''\,.
\end{multline*}
By \eqref{eq3.18*}, we have that $|\pe\vp-\vp|\ll t$ in the support of $\Psi$.   Thus, $III_1'$,
upon distributing $\nabla_x$ over $u_1^2, 1/J$, and $\Psi^2$,  yields integrals that may be handled just like the terms $J$, $S_{12}'$ and $K$, respectively, in \eqref{eq2.5}.
To handle $III_1''$, we first note that 
\begin{equation}\label{eq3.52}
\int_{\rn}\left(\int_0^\infty|(\pe-I) F|^2 \frac{dt}{t^3}\right)^{p/2}\,dx \lesssim \|\nabla F\|_{L^p(\rn)}^p\,,
\end{equation}
as may be seen by the use of the elementary identity
$\pe-I = \int_0^{\eta t} \partial_s P^*_s ds$,
along with Hardy's inequality in $t$, to reduce matters to \eqref{eq1.katovertp}.  We further note that by \eqref{eq3.18*} and the definition of $u_1$,
$$\sup_{t>\delta(x)/\eta}|u_1(x,t)| \leq \sup_{t>0}|u(x,t)| \leq N_*^\alpha(u)(x)\,.$$
Consequently, 
\begin{multline*}
III_1''
\lesssim\int_{2Q} \Big(N_*^\alpha(u)(x)\Big)^2 
\left(\int_0^\infty|(\pe-I) \vp|^2 \frac{dt}{t^3}\right)^{1/2}
\left(\int_0^\infty|t^2\partial_t\,\pe L_\|^*\vp|^2 \,\frac{dt}{t}\right)^{1/2}\!dx\\[4pt]
\lesssim \,  \left(\int_{2Q} \Big(N_*^\alpha(u)(x)\Big)^{2(2+\eps)/\eps}\right)^{\eps/(2+\eps)}  
\|\nabla \vp\|_{2+\eps}^2\, \lesssim\, (\gamma\lambda)^2 |Q|\,,
 \end{multline*}
where we have used \eqref{eq1.katovertp}, \eqref{eq3.52}, \eqref{eq1.hodge*},
and \eqref{eq3.x0}
(with $p_0:= 2(2+\eps)/\eps$.)

It remains now only to treat term $III_2$.
To this end, we use the Hodge decomposition \eqref{eq1.hodge} to write
$${\bf b}1_{5Q} = A_\|\nabla \vpt + \widetilde{\bf h}=
\left(A_\|\nabla_x \vpt-A_\| \nabla_x \Pe\vpt\right) \,+\,A_\| \nabla_x \Pe\vpt\,+\, \widetilde{\bf h}\,,$$
where $\Pe:= e^{-(\eta t)^2L_\|}$, and where $\widetilde{\bf h}$ is divergence free.  We recall that by construction, the various estimates that we have used for $\vp$ and $\pe\vp$, 
hold also for $\vpt$ and $\Pe\vpt$.
The contribution of $\widetilde{\bf h}$ may then be handled exactly like $II_2$ above,
while the contribution of $A_\| \nabla_x \Pe\vpt$ may be handled like $II_1$ above, i.e.,
by integrating by parts in $x$ to move $\nabla_x$ away from $\partial_t \pe\vp$.
Finally, the contribution of  $\left(A_\|\nabla_x \vpt-A_\| \nabla_x \Pe\vpt\right)$ in term $III_2$
equals
$$\iint_{\mathbb{R}^{n+1}_+}u_1^2\,\frac1J
\,\langle\,\left(\nabla_x \vpt- \nabla_x \Pe\vpt\right)  
 \,,\,A_\|^*\nabla_x\partial_t\pe\vp\rangle\, \Psi dt dx\,,$$
 which can then be handled like $III_1$.
 
 \subsection{Step 3:  from large $p$ to arbitrary $p$} 
 At this point, we observe that our work in the previous two subsections yields the 
 $S<N$ bound \eqref{eq1.SN}, for all finite $p>p_0$, where as above $p_0=2(2+\eps)/\eps$,
 and $2+\eps$ is the exponent in the Hodge decomposition \eqref{eq1.hodge}-\eqref{eq1.hodge**}
 (cf. \eqref{eq3.12++}.)
 We now proceed to remove the restriction on $p$, following \cite{FS}. Let us observe that 
 the standard pullback mechanism, as used in the proof of Corollary
 \ref{c2}, implies that on any Lipschitz graph domain
 $\Omega_\psi$ as in  \eqref{eq1.2*}, we obtain from \eqref{eq3.12++} the bound
\begin{equation}\label{eq3.65*}
\|S_\psi(u)\|_{L^p(\po_\psi)} \leq C_p \|N_{*,\psi}(u)\|_{L^p(\po_\psi)}\,,
\qquad p\left({\|\nabla\psi\|_\infty}\right)<p<\infty\,,
\end{equation}
for $Lu=0$ in $\Omega_\psi$, 
where $S_\psi(u),\,N_{*,\psi}(u)$ are the square function and non-tangential
maximal function relative to $\Omega_\psi$
(cf. \eqref{eq1.squarelip}-\eqref{eq1.ntlip}.)   For the moment, the range of $p$ depends upon the Lipschitz constant of $\psi$, because the ellipticity
of the pullback matrix depends upon this Lipschitz constant, and 
in turn,  the parameter $\eps$ that appears in
the Hodge decomposition, and in the definition of $p_0$, depends upon ellipticity.
The conclusion of Theorem \ref{th1} then follows immediately from \eqref{eq3.65*} and the following
 
 \begin{lemma} Suppose that for every Lipschitz graph domain $\Omega_\psi$, 
 and every elliptic $t$-independent matrix $A$ with real bounded 
 measurable coefficients, there exist constants $C$ and $q\in (0,\infty)$, 
 depending on dimension, ellipticity, and $\|\nabla\psi\|_\infty$, 
 such that   any solution  $u$ to
the equation $-{\rm div}_{x,t} A\nabla_{x,t}u=0$ in $\Omega_\psi$ satisfies 
\begin{equation}\label{lpap1}
\|S_\psi u\|_{L^q(\po_\psi)}\leq C \|N_{*,\psi} u\|_{L^q(\po_\psi)}\,.
\end{equation}
Then the $S<N$ estimate \eqref{eq1.SN} is valid for all $p\in (0,q)$.
 \end{lemma} 
 

\bp We follow the argument of \cite{FS}.    Set the aperture of the cone defining $N_*(u)$ to be 2.
Fix any $\lambda>0$ and let 
\begin{equation*}
F_\lambda:=\{x\in \rn:\,N_*u(x)\leq \lambda\}.
\end{equation*}

\noindent Then the distribution function $\tau_{N_*u}(\lambda):=|F_\lambda^c|$. 
Denote by ${\mathcal R}$ an (infinite) saw-tooth region above $F_\lambda$, i.e., 
${\mathcal R}={\mathcal R}(F_\lambda):=\cup_{x\in F_\lambda} \Gamma(x)$, where
the vertical cones $\Gamma(x)$ have aperture 1 and vertex at $x\in\rn$. 
Clearly, ${\mathcal R}$ is a Lipschitz graph domain (with boundary given by the graph of
$\psi(x):= \dist(x,F_\lambda)$), with Lipschitz constant 1, 
so, in particular, \eqref{lpap1} holds in ${\mathcal R}(F_\lambda)$ for some $q<\infty$.
Furthermore, we may take the cones defining $S_\psi$ and $N_{*,\psi}$ to have aperture 1/2.

Let $\tau_{S(u)}:=\{x\in\rn:S(u)>\lambda\}$, where we have fixed the aperture of the cone defining
$S(u)$ to be 1/2.  Then 
\begin{eqnarray*}
\tau_{Su}(\lambda)&=&|\{x\in F_\lambda:\, Su(x)>\lambda\}|+|\{x\in F_\lambda^c:\, 
Su(x)>\lambda\}|\nonumber \\[4pt]
&\leq & \frac{C}{\lambda^q}\int_{F_\lambda}\left(Su(x)\right)^q\,dx+\tau_{N_*u}(\lambda).
\end{eqnarray*}

\noindent However, due to \eqref{lpap1} on ${\mathcal R}(F_\lambda)$, 
\begin{eqnarray*}
\int_{F_\lambda}\left(Su(x)\right)^q\,dx &\leq& \int_{\partial{\mathcal R}(F_\lambda)}
\left(S_\psi u(x)\right)^q\,d\sigma (x)\leq \int_{\partial{\mathcal R}(F_\lambda)}\left(N_{*,\psi}u(x)\right)^q\,d\sigma(x)\nonumber \\[4pt]
&\lesssim &\int_{F_\lambda}\left(N_*u(x)\right)^q\,dx + \int_{\partial{\mathcal R}(F_\lambda)\setminus F_\lambda}\left(N_{*,\psi}u(x)\right)^q\,d\sigma (x),
\end{eqnarray*}
where $d\sigma$ is surface measure on the Lipschitz graph $t=\psi(x)$.
\noindent However, 
\begin{equation*}
\int_{F_\lambda}
\left(N_*u(x)\right)^q\,dx \leq C \int_0^\lambda t^{q-1}\tau_{N_*u}(t)\,dt. 
\end{equation*}

\noindent Furthermore, any point $x\in {\mathcal R}(F_\lambda)$ belongs to some cone with a vertex in $F_\lambda$. Since $N_*u\leq \lambda$ on $F_\lambda$, it follows that $|u(x)|\leq \lambda$ 
for any $x\in {\mathcal R}(F_\lambda)$, and therefore, $N_{*,\psi}u(x)\leq \lambda$ for any $x\in \partial  {\mathcal R}(F_\lambda)$. Hence, 
\begin{equation*}
\int_{\partial{\mathcal R}(F_\lambda)\setminus F_\lambda}\left(N_{*,\psi}u(x)\right)^q\,d\sigma(x) 
\leq C \lambda^q |\partial{\mathcal R}(F_\lambda)\setminus F_\lambda| \leq C \lambda^q 
|F_\lambda^c| = C \lambda^q \tau_{N_*u}(\lambda).
\end{equation*}
All in all, we have 
\begin{equation*}
\tau_{Su}(\lambda)\leq C \tau_{N_*u}(\lambda)+C 
\lambda^{-q}\int_0^\lambda t^{q-1}\tau_{N_*u}(t)\,dt.
\end{equation*}
Consequently, 
\begin{multline*}
\|Su\|_{L^p(\rn)}^p=C \int_0^{\infty}\lambda^{p-1}\tau_{Su}(\lambda)\,d\lambda \nonumber\\[4pt] 
\leq C \int_0^{\infty}\lambda^{p-1}\tau_{N_*u}(\lambda)\,d\lambda+
C \int_0^{\infty}\lambda^{p-q-1}\int_0^\lambda t^{q-1}\tau_{N_*u}(t)\,dtd\lambda 
\leq C \|N_*u\|_{L^p(\rn)}^p,
\end{multline*}
provided that $p<q$. \ep

\section{Proof of Theorem \ref{th2}:  local ``$N<S$" bounds}\label{sNS}
In this section, taking the global $S/N$ bounds,  as expressed in \eqref{eq1.10}
and \eqref{eq1.11}, as our starting point, we shall establish the local 
$N<S$ estimate as stated in Theorem \ref{th2},
following the proof of \cite[Theorem 3.18]{KKPT} very closely.
We shall prove Theorem \ref{th2} in the special case that the bounded
solution $u$ is continuous on the closure of $\reu$.  Of course, we shall obtain
the desired estimate \eqref{eq1.NSloc} with bounds depending only on dimension and ellipticity.
Eventually, in Section \ref{s-epp-app}, we shall see that, in order to prove Theorem \ref{th3},
it is enough to verify  \eqref{eq1.NSloc} in the sense of an {\it a priori} bound,
for solutions that are continuous up to the boundary.   On the other hand, {\it a posteriori},
with Theorem \ref{th3} in hand, the interested reader 
could revisit the arguments of the present section,
which continue to work with continuity at the boundary replaced by non-tangential convergence a.e. 
($dx$), to obtain Theorem \ref{th2} in the general case.  We omit the details, except to note that,
by the Fatou theorem of \cite{CFMS} (whose proof carries over, {\it mutatis mutandi}, 
to the case of non-symmetric coefficients),
a bounded solution has a non-tangential trace a.e. ($d\omega)$,
and thus, in the presence of Theorem \ref{th3}, also a.e. ($dx$).

Consider now a solution $u$ of the equation $Lu=0$ in $\reu$, which is
bounded and continuous on $\overline{\reu}$.  
We fix a cube $Q\subset \rn$, a constant $\theta\in (0,1)$, and recall that
$\theta Q$ is the cube concentric with $Q$, of side length $\theta\, \ell(Q)$.
We further fix constants 
$\theta_0,\theta_1...,\theta_6$ satisfying
$0<\theta<\theta_0<\theta_1<...
<\theta_6<1$.   Define a Lipschitz function $\psi:\rn\to
[0,\infty)$ such that $\|\nabla\psi\|_\infty\leq\eps_0$,
where $\eps_0$ is a small positive number to be chosen, 
$\psi\equiv 0$ on $\theta_0 Q$ and on $\rn\setminus Q$, and $\psi>0$
on $\theta_6Q\setminus \overline{\theta_0 Q}$.  In addition, we may suppose that
$\psi(x)\approx \ell(Q)$ on $\theta_5Q\setminus \theta_1Q$
(with the implicit constants depending on $\eps_0$).  In this section, we shall find it convenient
to work with the following variant of the non-tangential maximal function:
\begin{equation}\label{eq4.0-}
\widetilde{N}_*(w)(x)=\widetilde{N}^\gamma_*(w)(x):= 
\sup_{t>0} \frac1{|B_\gamma(x,t)|}\iint_{B_\gamma(x,t)}|w(y,s)|\,dy ds\,,
\end{equation} 
where $B_\gamma(x,t)$ is the ball with center $(x,t)$ and radius $\gamma t$, with $0<\gamma<1$.
We note that by Moser's interior estimates, if $Lu = 0$ in $\reu$, then
$N_*(u) \lesssim \widetilde{N}_*(u)$, pointwise, provided that the aperture of the cone defining
$N_*(u)$ is sufficiently small, depending on $\gamma$.

We recall that $R_Q$ is the ``short
Carleson box" above
$Q$ (cf. \eqref {eq1.cboxshort}), and we consider the domain $\Omega\subset\Omega_\psi$
(where $\Omega_\psi$ is the usual graph domain as in \eqref{eq1.2*}),  given by
$$\Omega:= \left\{(x,t):  x\in \theta_5 Q,\, \psi(x)<t<\psi(x) + \theta_5\ell(Q)/2\right\}\,.$$
We observe that $\Omega\subset R_Q$, provided that $\eps_0$ is chosen sufficiently small,
depending upon dimension and $\theta_5$.  
Let $K:= \po\setminus\{(x,\psi(x)): x\in \theta_1Q\}$.  We note that $K\subset\subset R_Q$, with
$\dist(K,\partial R_Q) \approx\ell(Q)$ (again provided that $\eps_0$ is small enough.) 
Let $\Phi_1\in C^\infty_0(\theta_2 Q)$, with $0\leq\Phi_1\leq 1$, and $\Phi_1 \equiv 1$ on 
$\theta_1 Q$.   We split $u=u_1+u_2$ in $\Omega$, where $Lu_i = 0$ in $\Omega$
and where $u_1,u_2$ are continuous and bounded in $\overline{\Omega}$, with
$$u_2\big|_{\po} = u(x,\psi(x)) \,\Phi_1(x)\,,$$
on $\{(x,\psi(x))\}\cap\po$, and zero otherwise on $\po$.   Note that
\begin{equation}\label{eq4.0}
\sup_\Omega |u_1|\leq\sup_{\po} |u_1|  \leq \sup_K |u|\,.
\end{equation} 
Consequently, 
\begin{equation}\label{eq4.1}\fint_{\theta Q} \widetilde{N}_{*,Q}(u_1)^2\, dx \leq \Big(\sup_K |u|\Big)^2\,,
\end{equation}
where the ``truncated" maximal function $\widetilde{N}_{*,Q}(u)$ is defined 
as in \eqref{eq4.0-}, except that we now consider a restricted supremum over $0<t\lesssim \ell(Q)$.
Moreover, by Fubini's theorem and Caccioppoli's inequality at the boundary,
\begin{equation}\label{eq4.2}\fint_{\theta Q} S_Q(u_1)^2(x) \,dx\lesssim
\sup_{0<t<c\ell(Q)}\fint_{\theta_0 Q}  |u_1(y,t)|^2 dy \lesssim \Big(\sup_K |u|\Big)^2\,,
\end{equation}
where $S_Q$ is defined with respect to cones $\Gamma_Q(x)$, which have been truncated
at height $\approx\ell(Q)$, so that $\Gamma_Q(x)\subset\Omega$, for $x\in \theta Q$, 
and where the implicit constants depend upon $\theta$ and $\theta_0$.

We now consider $u_2$.  Let $\Phi\in C^\infty_0(\theta_4 Q)$, with 
$0\leq\Phi\leq 1$, and $\Phi \equiv 1$
on $\theta_3 Q$.  Let $\mu \in C^\infty_0(\RR)$, with $\mu$ supported in $|t|<\theta_4 \ell(Q)/4$,
$\mu(t) \equiv 1$ for $|t|<\theta_4 \ell(Q)/8$, and set
$$v(x,t):= \Phi(x)\,\mu(t-\psi(x))\,u_2(x,t)\,.$$
As above, let $\Omega_\psi=\{t>\psi(x)\}$, 
and decompose $v=v_1+v_2$ in $\Omega_\psi$,
where $v_1$ is bounded and continuous in $\overline{\Omega_\psi}$, and 
solves
\begin{equation}
\begin{cases} Lv_1=0\quad \text{ in }\,\Omega_\psi\\ 
v_1\big|_{\po_\psi} = v\big|_{\po_\psi}\,,
\end{cases}\label{eq4.3}\end{equation}
while $Lv_2=Lv$ in $\Omega_\psi$, with $v_2|_{\po_\psi} = 0$.
We note that the solution $v_1$ may be constructed so that $v_1\to 0$ at infinity,
since its boundary data has compact support.
We now claim that there is a set $F\subset\Omega$,
with $\dist(F,\partial R_Q)\approx \ell(Q)$, such that
\begin{equation}\label{eq4.4}
\int_{\po_\psi}\left(\widetilde{N}_{*,\psi}(v_2)^2 +
S_\psi(v_2)^2\right)\, dx \lesssim |Q|\,\Big(\sup_F |u_2|\Big)^2\,,
\end{equation}
where $\widetilde{N}_{*,\psi},\,S_\psi$ are defined relative to $\Omega_\psi$
(cf. \eqref{eq4.0-} and \eqref{eq1.squarelip};  in the case of $\widetilde{N}_{*,\psi}$, the ball $B_\gamma(x,t)$ now has radius equal to $\gamma(t-\psi(x))$, with $\gamma$ sufficiently small depending
on $\|\nabla\psi\|_\infty$.)  Let us momentarily take this claim for granted.
By \eqref{eq1.11}, 
we have
\begin{equation*}
\int_{\po_\psi} \widetilde{N}_{*,\psi}(v_1)^2  \lesssim 
\int_{\po_\psi} S_{\psi}(v_1)^2 
\lesssim\int_{\po_\psi} S_{\psi}(v)^2\,+\,\int_{\po_\psi} S_{\psi}(v_2)^2\,,
\end{equation*}
where we have used the pointwise bound $\widetilde{N}_{*,\psi}(w)\leq N_{*,\psi}(w)$.
We observe that
\begin{equation*}\nabla v\, =\, \Phi(x)\,\mu(t-\psi(x))\,\nabla u_2(x,t)\,+\,
\nabla \Big(\Phi(x)\,\mu(t-\psi(x))\Big)\,u_2(x,t)
=: {\bf V}_1 +{\bf V}_2\,,
\end{equation*}
and in turn,
\begin{multline*}
\nabla \Big(\Phi(x)\,\mu(t-\psi(x))\Big)\\[4pt] =
\Big(\nabla_x\Phi(x)\,\mu(t-\psi(x)) - \Phi(x)\,\mu'(t-\psi(x))\,\nabla_x\psi(x),
\Phi(x)\,\mu'(t-\psi(x))\Big)\,.%
\end{multline*}
Thus,  $\nabla (\Phi(x)\,\mu(t-\psi(x)))$ (restricted to $\Omega_\psi$), 
and hence also ${\bf V}_2$, are supported in
\begin{multline}\label{eq4.6}
\Big\{(x,t): x\in \theta_4 Q \setminus\theta_3 Q, \,0<t-\psi(x)<\theta_4\ell(Q)/4\Big\}\\[4pt]
\cup \Big\{(x,t): x\in \theta_4 Q , \,\theta_4\ell(Q)/8<t-\psi(x)<\theta_4\ell(Q)/4\Big\}\,=:E_1\cup E_2\,.
\end{multline}
Consequently, there is a set $F\subset\Omega$, with $\dist(F,\partial R_Q)
\approx \ell(Q)$, such that $|{\bf V}_2| \lesssim \ell(Q)^{-1} \sup_F |u_2|,$ whence it follows that
\begin{equation}\label{eq4.7}
\int_{\po_\psi} S_{\psi}(v)^2 \,d\sigma\lesssim
 \int_{\theta_5 Q} S_Q(u_2)^2(x)\,dx \,+\,|Q|\,\Big(\sup_F |u_2|\Big)^2\,,
 \end{equation}
 provided that the constant $\eps_0$ (which controls $\|\nabla\psi\|_\infty$) is sufficiently small.
 Moreover
 $$\fint_{\theta Q} \widetilde{N}^{\gamma}_{*,Q}(u_2)^2(x)\,dx 
 \lesssim \frac1{|Q|}\int_{\po_\psi}\widetilde{N}^{\gamma}_\psi(v)^2 d\sigma + 
 \Big(\sup_F |u_2|\Big)^2\,,$$  
if $\gamma$ is sufficiently small.
Indeed, in that case, for $x\in \theta Q$, and $0<t\lesssim \ell(Q)$,  we have that 
$(u_2-v))1_{B_\gamma(x,t)}$ is supported in a region of Whitney type, i.e., so that $t\approx \ell(Q)$,
inside $\Omega$.
Gathering these estimates, we obtain
\begin{multline*}
\fint_{\theta Q} \widetilde{N}_{*,Q}(u_2)^2(x)\,dx\\[4pt] \lesssim  \,\Big(\sup_F |u_2|\Big)^2
\,+\,  \frac1{|Q|}\int_{\po_\psi} \widetilde{N}_{*,\psi}(v_1)^2 d\sigma\,
+ \, \frac1{|Q|}\int_{\po_\psi} \widetilde{N}_{*,\psi}(v_2)^2 d\sigma
\\[4pt] \lesssim  \,\Big(\sup_F |u_2|\Big)^2
\,+\, \fint_{\theta_5 Q} S_Q(u_2)^2(x)\,dx \,+\,\frac1{|Q|}\int_{\po_\psi} S_{\psi}(v_2)^2\,
+ \, \frac1{|Q|}\int_{\po_\psi} \widetilde{N}_{*,\psi}(v_2)^2\,\\[4pt]
\lesssim  \,\Big(\sup_F |u_2|\Big)^2
\,+\, \fint_{\theta_5 Q} S_Q(u_2)^2(x)\,dx\,,
\end{multline*}
where in the last step we have used the claim \eqref{eq4.4} (and where we have also used 
that the set $F$ 
may be taken to be the same in \eqref{eq4.4} and \eqref{eq4.7}:  just take the union of the two,
or, see the proof of \eqref{eq4.4} below.)
Combining the latter estimate with \eqref{eq4.0}-\eqref{eq4.2}, 
and setting $K_Q:= F\cup K$, we have
\begin{multline*}
\fint_{\theta Q}  \widetilde{N}_{*,Q}(u)^2(x)\,dx\lesssim
\fint_{\theta Q}  \widetilde{N}_{*,Q}(u_1)^2(x)\,dx+
\fint_{\theta Q}  \widetilde{N}_{*,Q}(u_2)^2(x)\,dx\\[4pt]
\lesssim  \,\Big(\sup_{K} |u|\Big)^2\,+ \,\Big(\sup_F |u_2|\Big)^2
\,+\, \fint_{\theta_5 Q} S_Q(u_2)^2(x)\,dx\\[4pt]
\lesssim  \,\Big(\sup_{K_Q} |u|\Big)^2\,+ \,\Big(\sup_F |u_1|\Big)^2\,+\,
\fint_{\theta_5 Q} S_Q(u)^2(x)\,dx\,+\,\fint_{\theta_5 Q} S_Q(u_1)^2(x)\,dx\\[4pt]
\lesssim  \,\Big(\sup_{K_Q} |u|\Big)^2\,+ \,
\fint_{\theta_5 Q} S_Q(u)^2(x)\,dx\,,
\end{multline*}
whence \eqref{eq1.NSloc}, the conclusion of Theorem \ref{th2}, follows directly.

It remains to prove the claim \eqref{eq4.4}.  To this end, 
we shall require the following lemma.
For notational convenience, 
we write $X=(x,t)$ to denote points in $\ree$.
\begin{lemma}\label{l4.8}  Let $x_0\in \rn$, $r>0$, and set $X_0:=(x_0,0)$ and
$B:= B(X_0,r)$.  Let $\kappa B$ denote the concentric dilate of $B$ by a factor of $\kappa$.
Suppose that $w$ is bounded and continuous on $\overline{\reu}$,
with $w\to 0$ at infinity, that
$Lw=0$ in $\reu\setminus B$, and that $w|_{\rn\setminus B} \equiv 0$.
Set $$\m:= \|w\|_{L^\infty\left(\reu\cap (3B\setminus 2B\right)}\,.$$
Then there exist constants $C$ and $\nu >0$, depending only upon dimension and ellipticity,
such that
$$|w(X)|\leq C \m
\left(\frac{r}{|X-X_0|}\right)^{n-1+\nu}\,,\qquad |X-X_0|\geq 3r\,.$$
\end{lemma}

\begin{remark} We note that in the case that $L$ is symmetric, Lemma 4.9 is a 
well-known classical result of Serrin and Weinberger \cite{SW}.  
However, their proof does not carry over
to the non-symmetric case, therefore we shall supply a proof below.
\end{remark}

We defer for the moment the proof of the lemma.

Recall that $Lv_2=Lv$ in $\Omega_\psi$, 
that $Lu_2=0$ in $\Omega$, and that $\Phi(x)\mu(t-\psi(x))1_{\Omega_\psi}(x,t)$ is 
supported in $\Omega$.  Therefore,
\begin{equation}\label{eq4.9*}
Lv_2 = \dv\Big( A\nabla\big(\Phi \, \mu\big)\, u_2\Big) + \nabla(\Phi\mu)\cdot
A\nabla u_2\,=:\dv {\bf f} + g\,.
\end{equation}
Recall also that $\nabla(\Phi\mu)$ (restricted to $\Omega_\psi$) 
is supported in the union $E_1\cup E_2$ of the sets defined in \eqref{eq4.6}.  We observe that,
by construction,  $u_2|_{\po}$ is supported in $\po\cap\po_\psi$, and
$\supp u_2(x,\psi(x))\subset \theta_2 Q$, while
$E_1\cup E_2\subset \Omega$, with $\dist(E_2,\po\cap\po_\psi)\approx \ell(Q)$, and 
$1_{E_1}(x,\psi(x) \subset \theta_4 Q\setminus\theta_3 Q$.  Therefore,
by Caccioppoli's inequality at the boundary,
we have that
\begin{multline}\label{eq4.10*}
\iint|g|^2 dx dt \lesssim \ell(Q)^{-2}\iint_{E_1\cup E_2}|\nabla u_2|^2 dx dt \lesssim 
\ell(Q)^{-4} \iint_{E_1^*\cup E_2^*}|u_2|^2 dx dt\\[4pt]
\lesssim \ell(Q)^{n-3}\,\Big(\sup_F |u_2|\Big)^2\,,
\end{multline}
where $E_i^*\subset \Omega$ is a slightly fattened version of $E_i$, with
$ E^*_1\cup E^*_2\subseteq F \subset\Omega$, and $F\subset \subset R_Q$,
with $\dist(F,\partial R_Q)\approx\ell(Q)$.
Moreover, we have that
\begin{equation}\label{eq4.11*}
\|{\bf f}\|_{\infty} \lesssim \ell(Q)^{-1} \sup_F |u_2|\,.
\end{equation} 
Since $v_2$ vanishes on $\po_\psi$,
it follows from \eqref{eq4.9*} that
$$v_2 = L_D^{-1}\Big(\dv {\bf f} + g\Big)\,,$$
where $L_D$ is the operator $L$ with Dirichlet boundary condition in $\Omega_\psi$.
Now, $\nabla L_D^{-1} \dv$ is bounded on $L^2(\Omega_\psi)$, and
$\nabla L_D^{-1}:\, L^{2_*}(\Omega_\psi)\to L^2(\Omega_\psi)$,
where $2_*:= (2n+2)/(n+3)$ is the $(n+1)$-dimensional Sobolev exponent. 
Therefore, since ${\bf f}$ and $g$ are supported in $\Omega\subset R_Q$, we have
\begin{multline}\label{eq4.12*}
\iint_{\Omega_\psi} |\nabla v_2|^2 \lesssim \iint_{\Omega} |{\bf f}|^2\,+\,
\left(\iint_{\Omega} |g|^{2_*}\right)^{2/2_*}\\[4pt]
\leq \,|R_Q|\, \|f\|^2_\infty  + |R_Q|^{-1+2/2_*} \iint |g|^2 \lesssim 
 \ell(Q)^{n-1}\,\Big(\sup_F |u_2|\Big)^2\,,
\end{multline}
where in the last step we have used \eqref{eq4.10*}-\eqref{eq4.11*}.
Consequently,
\begin{multline}\label{eq4.14}
\int_{\po_\psi} \left(S_\psi \left(v_2 \,1_{t\lesssim \ell(Q)}\right)\right)^2d\sigma
\approx \int_{\rn}\int_{\psi(x)}^{C\ell(Q)} |\nabla v_2(x,t)|^2\, \big(t-\psi(x)\big)\,dt dx\\[4pt]
\lesssim \ell(Q)\iint_{\Omega_\psi} |\nabla v_2|^2  \lesssim |Q|\,\Big(\sup_F |u_2|\Big)^2\,.
\end{multline}
Moreover,  since $v_2$ vanishes on $\po_\psi$, we have that 
\begin{multline}\label{eq4.15}
 \int_{\rn}\int_{\psi(x)}^{C\ell(Q)} | v_2(x,t)|^2\,dt dx =  
 \int_{\rn}\int_{\psi(x)}^{C\ell(Q)} \left| \int_{\psi(x)}^t \partial_sv_2(x,s) \,ds\right|^2\,dt dx \\[4pt]
 \lesssim \ell(Q)^2 \int_{\rn}\int_{\psi(x)}^{C\ell(Q)} \left| \nabla v_2(x,s) \right|^2\,ds dx
 \lesssim  \ell(Q)^{n+1}\,\Big(\sup_F |u_2|\Big)^2\,,
\end{multline}
by \eqref{eq4.12*}.  We let $x_Q$ denote the center of $Q$, and set
$r_Q= C_1 \ell(Q)$, with $C_1$ chosen large enough that $T_Q\subset
B_Q:=B(x_Q,r_Q)$.  Since $Lv_2 = 0 $ in $\reu\setminus B_Q$, and $v_2=0$
in $(\rn\times\{0\})\setminus B_Q$, by Moser's estimates we have that
\begin{equation}\label{eq4.16}
\m_Q:= \|v_2\|_{L^\infty\left(\Omega_\psi\cap(3B_Q\setminus 2B_Q)\right)}\lesssim |B_Q|^{-1/2}\,
\|v_2\|_{L^2\left(\Omega_\psi\cap 4B_Q\right)}\,\lesssim\, \sup_F |u_2|\,,
\end{equation}
where in the last step we have used \eqref{eq4.15}.
We observe that $-v_2=v_1$ in $\Omega_\psi\setminus B_Q$.
We may therefore apply Lemma \ref {l4.8} to $v_2$,  
with $r=r_Q$, $x_0=x_Q$, to obtain
\begin{multline}\label{eq4.17}
\int_{\po_\psi} \left(S_\psi \left(v_2 \,1_{t\gtrsim \ell(Q)}\right)\right)^2d\sigma
\approx \int_{\rn}\int_{C\ell(Q)}^{\infty} |\nabla v_2(x,t)|^2\, t\,dt dx\\[4pt]
\approx \sum_{k=k_0}^\infty2^k\ell(Q)\int_{2^k\ell(Q)}^{2^{k+1}\ell(Q)}\int_{\rn} |\nabla v_2(x,t)|^2\, dt dx
\\[4pt]\lesssim\, \sum_{k=k_0-1}^\infty\frac1{2^{k}\ell(Q)}
\int_{2^k\ell(Q)}^{2^{k+1}\ell(Q)} \int_{\rn}| v_2(x,t)|^2\, dxdt\\[4pt]
\lesssim \m_Q^2 \sum_{k=k_0-1}^\infty\frac1{2^{k}\ell(Q)}\left(\ell(Q)\right)^{2n-2+2\nu} 
\int_{2^k\ell(Q)}^{2^{k+1}\ell(Q)} \int_{\rn}| X-x_Q|^{-2(n-1+\nu)}\, dX\\[4pt]
\lesssim \, |Q|\, \Big(\sup_{F}|u_2|\Big)^2 
\sum_k2^{-k(n-2+2\nu)} \,\lesssim \,|Q|\, \Big(\sup_{F}|u_2|\Big)^2\,,
\end{multline}
since $n\geq 2$ and $\nu>0$, where in the third, fourth and fifth lines, respectively, 
we have used Caccioppoli's inequality, Lemma \ref {l4.8}, and \eqref{eq4.16}.   
Combining \eqref{eq4.14} and \eqref{eq4.17}, we produce the desired bound for
$S_\psi(v_2)$.   

We now turn to $\widetilde{N}_{*,\psi}(v_2)$.  By Lemma \ref {l4.8},
it is enough to establish \eqref{eq4.4} for $\widetilde{N}_{*,\psi,Q}(v_2)$, where the latter is defined by
restricting the supremum to values of $t\leq 3C_1\ell(Q)$.
To this end, we fix $(x,t) \in \Omega_\psi$, with $t\lesssim \ell(Q)$, and a ball $B_\gamma(x,t)$, centered at $(x,t)$, of radius $\gamma (t-\psi(x)$.  Our goal is to show that
\begin{equation}\label{eq4.18}
\frac{1}{B_\gamma(x,t)}\iint_{B_\gamma(x,t)}|v_2(y,\tau)| dy d\tau \lesssim 
M \left(\int_{\psi(\cdot)}^{C\ell(Q)}|\nabla v(\cdot,s)|\,ds \right)(x)\,,
\end{equation}
where $M$ denotes the Hardy-Littlewood operator acting in the ``horizontal"
(i.e., $x$) variable.  Momentarily taking \eqref{eq4.18} for granted, we find that
\begin{multline*}\int_{\po_\psi} \left(\widetilde{N}_{*,\psi,Q}(v_2)\right)^2\,d\sigma
\lesssim \int_{\rn} \left(M \left(\int_{\psi(\cdot)}^{C\ell(Q)}|\nabla v(\cdot,s)|\,ds \right)(x)\right)^2\,dx\\[4pt]
\lesssim \ell(Q) \int_{\rn}\int_{\psi(x)}^{C\ell(Q)}|\nabla v(\cdot,s)|^2\,ds dx
\lesssim \ell(Q)\iint_{\po_\psi}|\nabla v_2|^2\,  \lesssim \,|Q|\, \Big(\sup_{F}|u_2|\Big)^2\,,
\end{multline*}
as desired, by \eqref{eq4.12*}.  Turning to the proof of \eqref{eq4.18}, we observe that, since
$v_2$ vanishes on $\po_\psi$, the left hand side of \eqref{eq4.18} equals
$$\frac{1}{B_\gamma(x,t)}\iint_{B_\gamma(x,t)}\left|\int_{\psi(y)}^\tau\! \partial_sv_2(y,s)\, ds\right| 
dy d\tau\lesssim \fint_{|x-y|< C (t-\psi(x))}\int_{\psi(y)}^{C\ell(Q)}\!| \nabla v_2(y,s)|\, ds dy\,,$$
whence \eqref{eq4.18} follows immediately.  This concludes the proof of Theorem \ref{th2}
(for solutions that are continuous up to the boundary of $\reu$), 
modulo the proof of Lemma \ref{l4.8}.  

\begin{proof}[Proof of Lemma \ref{l4.8}]  Let us make several elementary reductions, as follows.
By dilation and translation invariance of the class of operators under consideration,
we may suppose that $B$ is the unit ball centered at 0, i.e., that $x_0=0$ and that $r=1$.
Furthermore, by renormalizing, we may suppose that $\m=1$, i.e., that
$|w|\leq 1$ on $\reu\cap (3B\setminus 2B)$.  Finally, we claim that without loss of generality,
we may suppose that $w\geq 0$.  Indeed, let $\Omega':= \reu\setminus 2B$
and set $f:= w|_{\po'}$.
Let $f=f^+-f^-$ be the splitting of $f$ into its positive and negative parts, and observe that
\begin{equation}\label{f+-}
\max(f^+,f^-) = |f| \leq \|w\|_{L^\infty\left(\reu\cap (3B\setminus 2B\right)}=\m =1\,,
\end{equation}
by our renormalization, since $f$ vanishes on $\po'\cap (\rn\times\{0\})$.
We then may construct  solutions $w_+,w_-$ in $\Omega'$, continuous up to the boundary
of $\Omega'$, with compactly supported
data $f^+, f^-$, respectively, which decay to 0 at infinity.  By the maximum principle,
$w=w_+-w_-$ in $\Omega'$, and furthermore, by \eqref{f+-}, we have that
$$\max(\|w_+\|_{L^\infty(\Omega')},\|w_-\|_{L^\infty(\Omega')})\leq \m=1\,.$$
Therefore, by treating separately $w_+,w_-$, we may suppose that $w$ is a
non-negative solution in $\Omega'$, with $\|w\|_{L^\infty(\Omega')}\leq1$.

Let $\Gamma(X,0)$ be
the fundamental solution for $L$ with pole at the origin,
so that $\Gamma(X,0)\approx |X|^{1-n}$ in $\ree\setminus\{0\}$.  
Set $w_0(X):=C_0\,\Gamma(X,0)$, where we choose the constant
$C_0$, depending only upon dimension and ellipticity, so that $w(X)\leq w_0(X)$ for
$X\in \reu\cap (3B\setminus 2B)$.    
By the decay of $w$ at infinity, it follows by the maximum principle that
$w(X) \leq w_0(X)$ for $X\in \reu\setminus 2B$.   We now make the following claim. 

\begin{claim} Suppose that
$w_1$ is continuous and bounded in $\overline{\reu\setminus 2B}$, with 
$w_1\geq 0$, $Lw_1=0$ in $\reu\setminus 2B$, $w_1\to 0$ at infinity,
and $w_1|_{\rn\setminus B}=0$.  Suppose further that 
$w_1(X) \leq w_0(X)$ for $X\in \reu\setminus 2^j B$, for some integer $j\geq 1$.
Then $w_1(X) \leq (1-\delta) \,w_0(X)$, in $\reu\setminus 2^{j+1}B$, for some $\delta>0$ 
depending only upon dimension and ellipticity.
\end{claim}

Since $w_0(X)\approx |X|^{1-n}$, the conclusion of Lemma \ref{l4.8} 
follows from the claim by a straightforward iteration argument,
whose details we omit.  Therefore, it remains only to establish the claim.
To this end, we fix $j$ such that
$w_1(X) \leq w_0(X)$ for $X\in \reu\setminus 2^j B$.
We note that by H\"older continuity at the boundary, 
and the fact that $w_0(Y)\approx 2^{j(1-n)}$ in $2^{j+2}B\setminus 2^jB$,
there is a constant $\eta_0$ depending only on ellipticity and dimension,
such that for $X=(x,0)$, with $|X|=2^{j+1}$, we have 
$$w_1(Y)\leq \frac12 w_0(Y)\,,\qquad \forall Y\in B(X,\eta_0 2^j)\cap\reu\,.$$
Set $h:= w_0-w_1$.  Then $h\geq 0$, $Lh=0$ in $\reu\setminus B$, and 
$$h(Y)\geq \frac12 w_0(Y)\,,\qquad \forall Y\in B(X,\eta_0 2^j)\cap\reu\,,$$
and for all $X=(x,0)$ with $|X|=2^{j+1}$.  Therefore, by Harnack's inequality,  
there is some constant $\delta>0$ depending only upon ellipticity and dimension such that
$$h(Y) \geq \delta \,w_0(Y)\,\qquad \forall Y\in \reu\,\,{\rm with}\,\, |Y|=2^{j+1}\,,$$
i.e., $w_1(Y) \leq (1-\delta) w_0(Y)$ for all $Y\in \reu$ with $|Y|=2^{j+1}$.  The claim now follows by the
maximum principle.
 \end{proof}

\section{$\epp$-approximability and the proof of Theorem \ref{th3}}\label{s-epp-app}

In order to prove Theorem \ref{th3}, it is enough, by \cite[Theorem 2.3]{KKPT}, to show that
if $u$ is bounded in $\reu$, with $\|u\|_\infty\leq 1$, and  $Lu=0$ in $\reu$, 
then $u$ enjoys the following ``$\epp$-approximability" property, for every $\epp>0$:
\begin{definition}\label{def5.1} Let $u\in L^\infty(\reu)$,
with $\|u\|_\infty \leq 1$.  Given $\epp>0$, we say that $u$ 
is $\epp$-{\bf approximable}
if for every cube $Q_0\subset \rn$, there is a $\vp =\vp_{Q_0}\in W^{1,1}(T_{Q_0})$ 
such that 
\begin{equation}\label{eq5.2a}\|u-\vp\|_{L^\infty(T_{Q_0})}<\epp\,,
\end{equation}
and 
\begin{equation}\label{eq5.2}
\sup_{Q\subset Q_0}\frac1{|Q|}\iint_{T_Q}|\nabla \vp(x,t)| \,dx dt \leq C_\epp\,,
\end{equation}
where $C_\epp$ depends also upon dimension and ellipticity, but not on $Q_0$.
\end{definition}
Actually, the definition of $\epp$-approximability given in \cite{KKPT}, is stated in
terms of the existence of a smooth, globally defined
$\vp$, but the version above is in fact all that is needed in the proof of Theorem 2.3
of that paper.  Moreover, the arguments of \cite{KKPT} do not require $\epp$-approximability
for all bounded solutions, but only for solutions whose boundary data is the characteristic function of
a bounded Borel set.  We shall return to this point below.

In this section, we shall assume that $u$ satisfies the following pair of estimates.
Given a cube $Q\subset \rn$, with center $x_Q$,
we let $P_Q:=(x_Q, (1-\eta) \ell(Q))$ denote the ``Corkscrew point" relative to 
$Q$, where $\eta>0$ is a small number to be chosen.  Note that,
if $\psi: \rn\to \re$ is Lipschitz, $\lVert \nabla \psi\rVert_\infty \leq M$, and 
$0\leq \psi\leq \frac{1}{8}\ell(Q)$ in $Q$, then 
$\dist\left(P_{Q}, \partial \Omega_\psi \cap T_Q\right)>\frac{\eta}{2}\,\ell(Q)$, 
provided that $\eta$ is sufficiently small.  Here, as usual,
$\Omega_\psi:=\left\{ (x,t):t>\psi(x) \right\}$.
\begin{est}\label{est1}
Let $Lu=0$ in $\reu, \lVert u\rVert_\infty<\infty$.  
We say that Estimate 1 holds if for every cube $Q\subset \rn$, and every $\psi$ as above, we have:
\begin{equation}
\int_{(1-s_n)Q}\lvert u\big(x,\psi(x)\big)-u(P_{Q})\rvert^2 dx\leq 
C_{M,\eta}
\iint_{\Omega_\psi\cap T_Q} \lvert \nabla u\rvert ^2 \,t\,dtdx\,,
\label{assumption}
\end{equation}
for some $s_n<1$ sufficiently small, where $C_{M,\eta}$ depends also on dimension and ellipticity.
\end{est}
\begin{est}\label{est2}
	Let $L$, $u$ be as in Estimate \ref{est1}, $\lVert u\rVert_\infty\leq 1$.  We say that
	Estimate 2 holds if
	\begin{equation}
		\sup_Q\frac1{|Q|}\iint_{T_Q} \lvert \nabla u(x,t)\rvert^2\,t\,dtdx\leq C\,.
		\label{est.2}
\end{equation}
\end{est}
\begin{remark}  For bounded null solutions of $t$-independent operators,
Estimate \ref{est2} has already been proved in general: 
indeed, it is simply a re-statement of Corollary \ref{c1}.
Moreover, at this point, we have verified 
Estimate \ref{est1} for solutions $u$ that are continuous up to the boundary.
Indeed, Estimate \ref{est1} follows easily from \eqref{eq1.12}, for every $s_n \in (0,1)$,
by interior estimates for solutions,  since $\psi\geq 0$ and thus $t-\psi(x)\leq t$.
In turn, by the pull-back mechanism described in the proof of Corollary
\ref{c2}, \eqref{eq1.12} for continuous $u$ follows directly from
\eqref{eq1.NSloc} for continuous $u$, 
and we have established the latter in Section \ref{sNS}.
As discussed at the beginning of Section \ref{sNS}, this will be enough
to establish Theorem \ref{th3}, as we shall see
momentarily.
\end{remark}

The main result in this section is:
\begin{theorem}\label{thm.a}
Assume that  $L u=0$ in $\reu, \lVert u\rVert_\infty \leq 1$, and that 
Estimate \ref{est1} and Estimate \ref{est2} hold for $u$.  Then, for each $\epp>0$, $u$ is $\epp$-approximable.  The constant $C_\epp$ in 
the Carleson measure condition \eqref{eq5.2} depends also on dimension, ellipticity
and the constants in Estimate \ref{est1}, Estimate \ref{est2}, but not on $Q_0$.
\end{theorem}
Before proving the theorem, let us use it to complete the proof of Theorem \ref{th3}.
\begin{proof}[Proof of Theorem \ref{th3}]
As noted above, in order to obtain the conclusion of
Theorem \ref{th3} via the program of \cite{KKPT}, 
it is enough to establish $\epp$-approximability for solutions with boundary data
of the form $u(x,0)=1_{\B}$, where $\B$ is a bounded Borel set.  Thus, given Theorem
\ref{thm.a}, it is enough to establish Estimate \ref{est1} for such solutions 
(since we already know that Estimate \ref{est2} holds for bounded solutions in general).
Moreover, it is enough to do this for a $t$-independent operator $L$
with smooth coefficients, as long as the bound in  \eqref{assumption} depends only upon 
the stated parameters.
Indeed, to prove Theorem \ref{th3}, we may then proceed initially under the
qualitative assumption that the coefficients are smooth, to
obtain the $A_\infty$ property of $L$-harmonic measure, but with
$A_\infty$ constants depending only on dimension and ellipticity.  We may then deduce the
$A_\infty$ conclusion in the general case
(i.e., without  {\it a priori} smoothness of the coefficients), 
by an approximation argument as in \cite[pp. 256-257]{KKPT}.

Therefore, we suppose that the coefficients of $L$ are smooth, 
and we fix a bounded Borel set $\B$.  
For $Y=(y,s)\in\reu$, set $u(Y):=\hm^Y(\B)$, the 
solution of the Dirichlet problem with data
$1_{\B}$.
Let us first suppose that $\B$ is open.
Let $X=(x,t)$ be a fixed point in  $\reu$.
By the inner  regularity of $L$-harmonic measure, and  Urysohn's lemma, 
we may find a sequence $\{f_k\}$ of continuous functions,  
and closed sets $F_1\subset F_2\subset...
\subset F_k\subset ...\subset \B$, 
such that $f_k\equiv 1$ on 
$F_k$, $f_k\equiv 0$ on $\B^c$, and such that 
$u_k(Y) \leq u(Y)$, for all $Y\in\reu$, 
with $u_k(X)\to u(X)$, as $k\to\infty$,
where $u_k$ denotes the solution with data $f_k$.
Thus, by Harnack's inequality, 
\begin{equation}\label{eq5.8}
u_k \to u\,,\quad {\rm uniformly\,\, on\,\, compacta \,\,in}\,\, \reu\,.
\end{equation} 

Our goal at the moment is to show that \eqref{assumption} holds for $u$.
 To this end, fix a small number $\delta>0$, and given a Lipschitz function $\psi$,
 we set $\psi_\delta(x):= \max(\psi(x),\delta)$.  We note that $\|\nabla\psi_\delta\|_\infty
 \leq \|\nabla\psi\|_\infty = M$, uniformly in $\delta$.
 Since \eqref{assumption} holds for solutions that are continuous up to the boundary,
we have for each $\delta>0$, and for every cube $Q$, that
 \begin{multline*}
 \int_{(1-s_n)Q}\lvert u\big(x,\psi_\delta(x)\big)-u(P_{Q})\rvert^2 dx =\lim_{k\to\infty}
\int_{(1-s_n)Q}\lvert u_k\big(x,\psi_\delta(x)\big)-u_k(P_{Q})\rvert^2 dx\\[4pt]\lesssim \limsup_{k\to\infty}
\iint_{\Omega_{\psi_\delta}\cap T_Q} \lvert \nabla u_k\rvert ^2 \,t\,dtdx\\[4pt]
\lesssim \iint_{\Omega_{\psi_\delta}\cap T_Q} \lvert \nabla u\rvert ^2 \,t\,dtdx
\,+\,\limsup_{k\to\infty}\iint_{\Omega_{\psi_\delta}\cap T_Q} \lvert \nabla (u_k-u)\rvert ^2 \,t\,dtdx\\[4pt]
\lesssim \iint_{\Omega_{\psi}\cap T_Q} \lvert \nabla u\rvert ^2\,,
 \end{multline*}
since $\Omega_{\psi_\delta}\subset\Omega_\psi$, 
where the implicit constants depend only upon the stated parameters, and
where the first limit holds by \eqref{eq5.8}, and the second
by Cacciopoli's inequality and \eqref{eq5.8}.  Recall that at this point we have assumed qualitatively
that our coefficients are smooth, so that $L$-harmonic measure and Lebesgue measure $dx$
on the boundary are mutually absolutely continuous.  Thus, by the results of \cite{CFMS}
(which, as we have observed above, remain valid in the setting of non-symmetric coefficients),
the bounded solution $u$ converges non-tangentially to its boundary data a.e. ($dx$).   We may
therefore take a limit as $\delta\to 0$ to obtain \eqref{assumption} for solutions
with boundary data given by the characteristic function of a bounded open set.  To establish
\eqref{assumption} when $u(x,0)=1_{\B}$ is the characteristic function of
a general bounded Borel set $\B$, we simply use outer regularity
of harmonic measure, and repeat the previous argument, but now with
$u_k(Y):= \hm^Y(\oo_k)$, where $\{\oo_k\}$ is a nested sequence of open sets
containing $\B$.  We omit the details. \end{proof}

We have now reduced matters to proving Theorem \ref{thm.a}.

\begin{proof}[Proof of Theorem \ref{thm.a}]
We fix a cube $Q_0\subset\rn$, and by dilation and translation invariance, we may suppose that
$Q_0=\left\{ 0\leq x_j\leq 1 \right\}$ is the unit cube in $\rn$.  Then
$T(Q_0):=T_{Q_0}=\left\{ 0\leq x_j\leq 1, 0\leq t\leq 1 \right\}$ is the associated Carleson box. 
For $N$ large (to be chosen to depend only on $n$) we let 
$S(Q_0):=\left\{ 0\leq x_j\leq 1, 2^{-N}\leq t\leq 1 \right\}$ be a ``rectangle''. 
As above, we let $P_{Q_0}=\left( \frac{1}{2}, \frac{1}{2}, \dots, \frac{1}{2}, 1-\eta \right)$ be the ``Corkscrew point" relative to $Q_0$, where $ 0<\eta<1/4$ is to be chosen later, 
and set $\widetilde{P}_{Q_0}=\left( \frac{1}{2}, \dots, \frac{1}{2}, 1 \right)$. 
Thus, $\lvert P_{Q_0}-\widetilde{P}_{Q_0}\rvert = \eta l(Q_0)=\eta$. 

For each $k=1,2,\dots$ we partition $Q_0$ into $2^{kNn}$ dyadic sub-cubes 
$Q_j^k$, with $l(Q_j^k)=2^{-kN}l\left( Q_0 \right)=2^{-kN}$. By abuse of language we call the 
collection $\left\{ Q_j^k \right\}_{j,k}$ ``the dyadic'' sub-cubes of $Q_0$
(of course, they are dyadic,  but they are not all of the dyadics). 
For $Q$ a ``dyadic'' sub-cube of $Q_0$, we define $T(Q), S(Q), P_Q, \widetilde{P}_Q$ analogously. Note that for each ``dyadic'' Q, the ``rectangles'' $S(Q')$ such that $Q'\subset Q$ and $Q'$ is ``dyadic'', form a ``Whitney'' tiling of $T(Q)$. For $\kappa>1$ near $1$ (depending on $N,n$) we let $\widetilde{S}(Q)$ be the rectangle obtained by expanding $S(Q)$ around its center by a factor of $\kappa$. 
If $\kappa$ is close enough to $1$ (depending on $N,n$), $Q=Q_j^k$, we still have $\text{dist} \left(\widetilde{S}(Q), \rn\times\{0\}\right)\approx 2^{-Nk}$. Moreover, we have

\begin{enumerate}
	\item [1)] $\left\{ \widetilde S(Q) \right\}$ have bounded overlap.
	\item [2)] If we fix $Q_1$, a ``dyadic'' cube, and consider $\left\{ S(Q) \right\}$, $Q\subset Q_1$, $Q$ ``dyadic'', then this is a ``Whitney'' tiling of $T(Q_1)$;  moreover,
$\left\{ \widetilde S(Q) \right\}$ are all contained in $T(\widetilde{Q_1})$ where $\widetilde{Q_1}$ is the $\kappa$ expansion of $Q_1$. 
We fix such a $\kappa$ from now on.
\end{enumerate}

We now fix an operator of the form $L=-\text{div} A(x)\nabla$, where $(x,t)\in \reu, x\in \rn$, 
and $A(x)$ 
is an $(n+1)\times (n+1)$ real, elliptic, $t$-independent matrix, not necessarily symmetric, with 
ellipticity constant $\lambda>0$. For solutions of $Lu=0$ in $\reu$,
we have the following classical estimates:

\subsection{Preliminary Estimates} For the reader's convenience, we state here some classical estimates that we shall use repeatedly, in the form that we shall use them, i.e., stated  
for ``dyadic" $Q$.

\medskip

\noindent (Cacciopoli:)
\begin{equation}
	\iint_{S(Q)}\lvert \nabla u\rvert^2 \leq \frac{C_{\lambda,n,N,\kappa}}{l(Q)^2}
	\iint_{\widetilde S(Q)}|u|^2
	\label{eqn.3}
\end{equation}

\noindent (Regularity:)
\begin{equation}
	\text{For } x,y\in S(Q), \,\,\lvert u(x)-u(y)\rvert \leq C_{\lambda,n,N,\kappa}
	\left( \frac{\lvert x-y\rvert}{l(Q)} \right)^\alpha\cdot l(Q)\left(\left|\widetilde S(Q)\right|^{-1}
	\iint_{\widetilde S(Q)}\lvert \nabla u\rvert^2 \right)^{\frac{1}{2}},
	\label{eqn.4}
\end{equation}
$\alpha=\alpha(\lambda,n)>0.$

\noindent(Regularity bis:)
\begin{equation}
	\text{For } x,y\in S(Q), \,\,\lvert u(x)-u(y)\rvert \leq C_{\lambda,n,N,\kappa}
	\left( \frac{\lvert x-y\rvert}{l(Q)} \right)^\alpha\left(\left|\widetilde S(Q)\right|^{-1}
	\iint_{\widetilde S(Q)}\lvert  u\rvert^2 \right)^{\frac{1}{2}},
	\label{eqn.4'}
\end{equation}
$\alpha=\alpha(\lambda,n)>0.$

We now return to the proof of Theorem \ref{thm.a}.	
Fix $Q$ ``dyadic'', $Q\subset Q_0$, $\widetilde P_Q$ as before. Let $\widetilde Q$ be the interval (in $\rn\times \{l(Q)\}$) centered at $\widetilde P_Q$, with $\text{diam}(\tilde Q)=2\eta l(Q)$, 
	so that $\widetilde Q\subset \text{top} \,S(Q)$. 
	Note that $H^{n}(\widetilde Q)=c_n\eta^{(n)}\lvert Q\rvert$.

\begin{claim} \label{c5.12}
Assume $\epp>0$ is given, Assume that for some constant $A$, $\lvert A\rvert \leq 1$, we have $\lvert u(P_Q)-A\rvert \geq \frac{\epp}{10}$. Then $\forall X\in \widetilde Q$, we have 
	$\lvert u(X)-A\rvert \geq \frac{\epp}{20}$, 
	provided $\eta=\eta(\epp,\lambda,n)$ is small enough. \end{claim}

	Indeed, by (\ref{eqn.4'}), $$\lvert u(X)- u(P_Q)\rvert \leq C_{\lambda,n,N,\kappa}
	\left( \frac{\lvert X-P_Q\rvert}{l(Q)} \right)^\alpha\cdot\left(\aaviint_{\!\!\widetilde S(Q)}u^2 \right)^{\frac{1}{2}}\leq C_{\lambda,n,N,\kappa}\cdot \eta^\alpha\leq \frac{\epp}{20}$$ if $X\in \widetilde Q$, 
	and if $\eta$ is small.

	Now, given $\epp>0$ as in Thm \ref{thm.a}, we choose and fix $\eta$ as in Claim
	\ref{c5.12}.

	\subsection{Stopping Time Construction, part I} 
	
	We will now define ``generation''  cubes. We set $G_0=\left\{ Q_0 \right\}$. Fix $\epp>0$, 
	and define the first generation, $G_1=G_1(Q_0)$ to be the maximal ``dyadic'' $Q\subset Q_0$, for which $\lvert u(P_Q)-u(P_{Q_0})\rvert\geq \frac{\epp}{10}$. 
	The ``dyadic'' cubes 
in $G_1(Q_0)$ have pairwise disjoint interiors. For $Q\in G_1(Q_0)$ we define $G_1(Q)$ in the same way. We set $G_2=G_2(Q_0)=\cup\left\{ G_1(Q): Q\in G_1 \right\}$. Later generations, $G_3, G_4, \dots$ are defined inductively. Note that each $Q\in G_{p+1}$ is contained in a unique $Q'\in G_p$ and $\lvert u(P_Q)-u(P_{Q'})\rvert \geq \frac{\epp}{10}.$

\begin{lemma}
	There exist $0<\mu<1$, and $N=N(\lambda,n,\mu)$ such that 
	\begin{equation*}
		\sum_{Q_j\in G_1}\lvert Q_j\rvert \leq C_{\epp,\lambda,n,(\ref{assumption}),\,\mu}\iint_{T(Q_0)\backslash\cup_{Q_j\in G_1}T(Q_j)} t\lvert \nabla u\rvert^2 dxdt+(1-\mu)\lvert Q_0\rvert\,,
	\end{equation*}
	\label{lma.1}
and more generally, if $Q'\in G_p$, we have
\begin{equation*}
	\sum_{Q_j\in G_1(Q')}\lvert Q_j\rvert \leq C_{\epp,\lambda,\mu,(\ref{assumption}),\,\mu}\iint_{T(Q')\backslash\cup_{Q_j\in G_1(Q')} T(Q_j)} t\lvert \nabla u\rvert^2 dxdt+(1-\mu)\lvert Q'\rvert.
\end{equation*}
\end{lemma}

\begin{proof}
	We prove the first estimate, the proof of the second one being the same. Consider the infinite
	downward cone,  $\Gamma_\delta:=\left\{ (x,t): |x|<-\delta t, t<0 \right\}$, where $\delta>0$
	is small. Let $U_1= \cup T(Q_j)$, $Q_j\in G_1$. Consider $\Omega_-=\left( \cup_{P\in U_1} (P+\Gamma_\delta) \right)\cap T(Q_0)$ and also $\Omega_+=\left( T(Q_0)\backslash\Omega_- \right)^\circ.$

We begin with several observation. If $Q\in G_1$, then $l(Q)\leq 2^{-N}$, by the definition of ``dyadic'' and the fact that $Q_0\notin G_1$. Also, $\Omega=\cup_{P\in U_1}(\Gamma_\delta+P)$ is a domain given as the domain below the graph of a Lipschitz function $\Psi_1$, whose Lipschitz constant is less than $\frac{1}{\delta}$. (One way to see this is that $\Omega$ verifies the uniform infinite exterior and interior cone conditions with respect to uniform vertical cones, since $U_1$ is given by a graph). The next observation is that, for $N>2, 0\leq \Psi_1\leq \frac{1}{4}$ on $Q_0$. Another observation is that $\Omega_+\cap U_1=\emptyset$. Let $Q_{i_0}, Q_{i_1}\in G_1$ be given. We say that ``$Q_{i_0}$ partially covers'' 
$Q_{i_1}$ if $Q_{i_0}\neq Q_{i_1}$, and 
\begin{equation*}
	\left[ \left[ \cup_{P\in T(Q_{i_0})}(\Gamma_\delta+P) \right] \cap T(Q_{i_0})) \right] \cap \text{top} 
	\,T(Q_{i_1}) \neq \emptyset,
\end{equation*}
where we note that $\text{top}\, T(Q_{i_1})=\text{top} \,S(Q_{i_1})$, and $\cup_{P\in T(Q_{i_0})}(\Gamma_\delta+P)\cap T(Q_0)=\cup_{P\in \text{top}\, T(Q_{i_0})}(\Gamma_\delta+P)\cap T(Q_0)$.

Note that if $Q_{i_0}$ partially covers $Q_{i_1}$, we must have $l(Q_{i_1})<l(Q_{i_0})$.

We say that $Q_{i_0}, Q_{i_1}, \dots, Q_{i_k}\in G_1$ are such that 
$(Q_{i_0}, Q_{i_1}, \dots, Q_{i_k})$ forms a chain starting at $Q_{i_0}$ and ending at $Q_{i_k}$, 
if for each $0\leq j\leq k-1$, $Q_{i_j}$ partially covers $Q_{i_{j+1}}$. Fix $Q_{i_0}\in G_1$. We define $T_r(Q_{i_0})$, the tree with  $\text{top } Q_{i_0}$, by
\begin{multline*}
T_r(Q_{i_0}):=\\[4pt]
\left\{ \text{all intervals }Q_j\in G_1: \text{there exists a chain starting at } Q_{i_0}, \text{ending at } Q_j \right\}\cup Q_{i_0}.
\end{multline*}
Finally, we say that $Q_{j_0}\in G_1$ is ``uncovered'' if there exists no $Q\in G_1$ with $Q$ partially covering $Q_{j_0}$.

\begin{fact}\label{fact.1}
For $\delta$ small, $T_r (Q_{j_0})\subset 8Q_{j_0}$, for any $Q_{j_0}\in G_1$, where $8Q_{j_0}$ is the cube with length $8l(Q_{j_0})$ and same center as $Q_{j_0}$.
\end{fact}

\begin{proof}
	Let $Q_j\neq Q_{j_0}\in T_r(Q_{j_0})$. Then there exists a chain $\left( Q_{j_0}, Q_{j_1}, \dots, Q_{j_k} \right)$ with $Q_{j_k}=Q_j$. Note that since $Q_{j_s}$ partially covers 
$Q_{j_{s+1}}, s\geq 0, l(Q_{j_{s+1}})\leq 2^{-N}l(Q_{j_s}) $. Also note that if $Q_{i_0}$ partially 
covers $Q_{i_1}, \exists y_1\in Q_{i_1}$ with
	\begin{equation*}
		\left( y_1, l(Q_{i_1}) \right)\in \left[ \cup_{P\in T(Q_{i_0})}(\Gamma_\delta+P)\cap 
T(Q_0) \right]=\left[ \cup_{P\in\text{top} \,T(Q_{i_0})}(\Gamma_\delta+P)\cap T(Q_0) \right].
	\end{equation*}
	Since $(y_1, l(Q_{i_1}))\in T(Q_0)$, (because $Q_{i_1}\subset Q_0$), there exists 
	$P=\left( x,l(Q_{i_0}) \right), x\in Q_{i_0}$ such that $(y_1, l(Q_{i_1}))\in \Gamma_\delta+P$, 
	i.e., $\lvert y_1-x\rvert<-\delta \left[ l(Q_{i_1})-l(Q_{i_0}) \right]
	=\delta\left[ l(Q_{i_0})-l(Q_{i_1}) \right].$

	Let $x_{i_0}=$ center of $Q_{i_0}$, $r_{i_0}=\max_{x\in Q_{i_0}}\lvert x-x_{i_0}\rvert$, 
	so that $r_{i_0}=c_n l(Q_{i_0}).$ For any $\tilde y\in Q_{i_1}$,
	\begin{multline*}
		\lvert x_{i_0}-\tilde y\rvert \leq \lvert x-x_{i_0}\rvert
		+\lvert x-y_1\rvert +\lvert y_1-\tilde y\rvert \\[4pt]
\leq c_n l(Q_{i_0})+\delta\left[ l(Q_{i_0})-l(Q_{i_1}) \right]+2c_n l(Q_{i_1})\leq 4c_nl(Q_{i_0}),
	\end{multline*}
	if $\delta$ is small.

	Next note that if $(Q_{j_0}\dots,Q_{j_k})$ is a chain, then $l(Q_{j_k})\leq 
	2^{-kN}l(Q_{j_0}).$ Suppose now that $\tilde y\in Q_{j_k}=Q_j$. 
Then, $\lvert \tilde y-x_{j_{k-1}}\rvert \leq 4c_n l(Q_{j_{k-1}})$ by the previous estimate and
	\begin{equation*}
		\lvert x_{j_s}-x_{j_{s-1}}\rvert \leq 4c_n l(Q_{j_{s-1}}), s=1,\dots,k
	\end{equation*}

	Hence,
	\begin{eqnarray*}
		\lvert \tilde y-x_{j_{0}}\rvert &\leq& \lvert \tilde y-x_{j_{k-1}}\rvert+\lvert x_{j_{k-1}}-x_{j_{k-2}}\rvert+\dots+\lvert x_{j_1}-x_{j_0}\rvert\\
		&\leq& 4c_nl(Q_{j_{k-1}})+4c_nl(Q_{j_{k-2}})+\dots+4c_nl(Q_{j_0})\\
		&\leq&4c_nl(Q_{j_0})\left[ 1+\frac{1}{2^N}+\frac{1}{2^{2N}}+\dots+\frac{1}{2^{(k-1)N}} \right]\\
		&\leq&8c_nl(Q_{j_0}),
	\end{eqnarray*}
	where we used $l(Q_{j_s})\leq 2^{-Ns}l(Q_{j_0})$, which follows because $(Q_{j_0},\dots,Q_{j_s})$ is a chain. Fact \ref{fact.1} follows.
\end{proof}

\begin{fact}\label{fact.2}
\begin{equation*}
	\lvert \cup_{Q\in T_r(Q_{j_0})} Q\rvert \leq c_n\lvert Q_{j_0}\rvert.
\end{equation*}
\end{fact}
Follows from Fact \ref{fact.1} and the disjointness of the intervals in $G_1$.

\begin{fact}\label{fact.3}
Assume that $Q_{j_0}\in G_1$ is ``uncovered''. Then, $(x,l(Q_{j_0})), x\in Q_{j_0}$, belongs to the graph of $\Psi_1$, i.e. $\Psi_1(x)=l(Q_{j_0}), \,x\in Q_{j_0}$, and hence to the boundary of $\Omega_+\cap T(Q_0)$.
\end{fact}

This is immediate from the definition of $\Omega_-, \Psi_1$ and the definition of ``partially covers'' and ``uncovered''.

Let now $\tilde G_1=\left\{ Q\in G_1: Q \text{ is ``uncovered''} \right\}$.

\begin{fact}\label{fact.4}
	\begin{equation*}
		G_1=\cup_{Q\in \tilde G_1} T_r(Q).
	\end{equation*}
\end{fact}

It suffices to show $G_1\subset \cup_{Q\in \tilde G_1}T_r(Q)$. Define $i_1=
\min\left\{ i: l(Q)=2^{-iN}, Q\in G_1 \right\}$. Let $G_{1,1}=G_1, \tilde G_{1,1}=
\left\{ Q\in G_1: l(Q)=2^{-i_1N} \right\}$. Note that if $Q\in \tilde G_{1,1}$, then $Q$ is ``uncovered'', because if $Q'$ partially covers $Q$, $l(Q)\leq 2^{-N}l(Q')$ which is impossible for $Q\in \tilde G_{1,1}$ since $l(Q)$ is maximal among lengths in $G_1$. Note also that $i_1\geq 1$. Next, let $G_{1,2}=G_1\backslash \cup_{Q\in \tilde G_{1,1}}T_r(Q)$. Let $i_2=\min\left\{ i: l(Q)=2^{-iN}, Q\in G_{1,2} \right\}$, unless $G_{1,2}=\emptyset$, in which case the process stops. Note that unless the process stops, $i_2>i_1$. Let now $\tilde G_{1,2}=\left\{ Q\in G_{1,2}: l(Q)=2^{-i_2N} \right\}$. We claim that if $Q_1\in \tilde G_{1,2}$, then $Q_1$ is ``uncovered''. Suppose not, let $Q'\in G_1$ partially cover $Q_1$. Then, 
$l(Q_1)<l(Q')$, so $Q'$ cannot belong to $G_{1,2}$. Hence, 
$Q'\in \cup_{Q\in \tilde G_{1,1}}T_{r}(Q)$. Thus, there exists $Q\in \tilde G_{1,1}$ such that $Q'\in T_r(Q)$, i.e., $\exists \left( Q_{i_0},\dots,Q_{i_k} \right)$ a chain, with $Q=Q_i\in \tilde G_{1,1}, Q_{i_k}=Q'$. 
But then, since $\left( Q_{i_0},\dots,Q_{i_k},Q_1 \right)$ is a chain, $Q_1\in T_{r}(Q_{i_0}), Q_0\in 
\tilde G_{1,1}$, which contradicts the fact that $Q_1\in G_{1,2}$. Thus, $Q_1$ is ``uncovered''. Next, we define
\begin{equation*}
	G_{1,3}=G_{1,2}\backslash \bigcup_{Q\in \tilde G_{1,2}} T_r(Q) 
	= G_1\backslash \left[ \cup_{Q\in \tilde G_{1,1}}T_r(Q) \bigcup \cup_{Q\in \tilde G_{1,2}}T_r(Q) \right].
\end{equation*}

Let $i_3=\min\left\{ i: l(Q)=2^{-iN}, Q\in G_{1,3} \right\}$ (unless $G_{1,3}=\emptyset$ in which case the process stops). If the process does not stop, we let $\tilde G_{1,3}=\left\{ Q\in G_{1,3}:l(Q)
=2^{-i_3N} \right\}$. We claim that if $Q_1\in \tilde G_{1,3}$ then $Q_1$ is ``uncovered''. If not, 
$\exists Q'\in G_1$, with $Q'$ partially covering $Q_1$, so that $l(Q_1)<l(Q')$. Hence, $Q'$ cannot belong to $G_{1,3}$. If $Q'\in \cup_{Q\in \tilde G_{1,2}}T_r(Q)$, we reach a contradiction as before. Hence, $Q'$ cannot belong to $G_{1,2}$, since $G_{1,2}=G_{1,3}\cup 
\left[ \cup_{Q\in \tilde G_{1,2}}T_r(Q) \right]$. Since $G_1=G_{1,2}\bigcup 
\cup_{Q\in \tilde G_{1,1}}T_r(Q), Q'\in \cup_{Q\in \tilde G_{1,1}}T_r(Q)$. But then 
$Q_1\in \cup_{Q\in \tilde G_{1,1}}T_r(Q)$, a contradiction. We continue inductively in this manner. If the process stops at stage $k$, we have
\begin{equation*}
	G_1\subset \bigcup_{Q\in \tilde G_{1,k-1}}T_r(Q)\cup\bigcup_{Q\in \tilde G_{1,k-2}}T_r(Q)\cup\dots\cup\bigcup_{Q\in\tilde G_{1,1}}T_r(Q),
\end{equation*}
and Fact \ref{fact.4} follows. If the process never stops, $i_k\uparrow \infty$ and it is also easy to verify Fact \ref{fact.4}.

\begin{fact}\label{fact.5}
	\begin{equation*}
		\Sigma_{Q_j\in G_1}\lvert Q_j\rvert \leq c_n \sum_{Q_j\in 
		\tilde G_1}\lvert Q_j\rvert,\quad\quad\quad\quad (c_n>1).
	\end{equation*}
\end{fact}
Let $O_1=\cup_{Q_j\in G_1}Q_j$, $\lvert O_1\rvert=\sum_{Q_j\in G_1}\lvert Q_j \rvert$. 
Now use Fact \ref{fact.4} and Fact \ref{fact.2}.

\emph{End of the proof of Lemma \ref{lma.1}}: For $\mu$ to be chosen, $N$ to be chosen, consider now:

\begin{case}
	\begin{equation*}
		\sum_{Q_j\in G_1}\lvert Q_j\rvert \leq (1-\mu)\lvert Q_0\rvert.
	\end{equation*}
	In this case Lemma \ref{lma.1} clearly holds.
\end{case}

\begin{case}\label{case.2}
	\begin{equation*}
		\sum_{Q_j\in G_1}\lvert Q_j\rvert > (1-\mu)\lvert Q_0\rvert.
	\end{equation*}
\end{case}

	Consider now $\tilde G_1'=\left\{ Q_j\in \tilde G_1: 
Q_j\subset (1-s_n)Q_0 \right\}$. Let 
$$\tilde G_1'':=\left\{ Q_j\in \tilde G_1: Q_j\cap \left( Q_0\backslash (1-s_n)Q_0 \right)\neq 
\emptyset \right\}.$$ We claim that if $Q_j\in \tilde G_1''$, then, if $N$ is large enough, $Q_j\cap (1-2s_n)Q_0=\emptyset$. Let $x_0=$ center of 
$Q_0, x_1\in Q_j\cap (Q_0\backslash (1-s_n)Q_0), x\in Q_j$. Then, 
$$\lvert x-x_0\rvert \geq \lvert x_1-x_0\rvert-\lvert x-x_1\rvert \geq d_n(1-s_n)l(Q_0)-2d_n2^{-N}l(Q_0)
\geq d_n(1-2s_n)l(Q_0),$$ 
for $N$ large, where $d_n$ is chosen so that if $Q$ 
is a cube with center $x_Q$ and length $l(Q)$, then, for $x\in Q$, $\lvert x-x_Q\rvert \leq d_n l(Q)$.

	Because of the claim, $\sum_{Q_j\in \tilde G_1''}\lvert Q_j\rvert \leq 
	\left[ 1-(1-2s_n)^n \right]\lvert Q_0\rvert$. But then, if we choose $s_n$ so 
	small that, with $c_n$ as in Fact \ref{fact.5}, we have $c_n\left[ 1-(1-2s_n)^n \right]\leq \delta_n$, where $2\delta_n<1$ and $\mu=\delta_n$, then

	\begin{eqnarray*}
		(1-\mu)\lvert Q_0\rvert &\leq& \sum_{Q_j\in G_1}\lvert Q_j\rvert \leq 
		c_n\sum_{Q_j\in \tilde G_1}\lvert Q_j\rvert\\
		&\leq&c_n\sum_{Q_j\in \tilde G_1'}\lvert Q_j\rvert + 
		c_n\sum_{Q_j\in \tilde G_1''}\lvert Q_j\rvert\\
		&\leq&c_n\sum_{Q_j\in \tilde G_1'}\lvert Q_j\rvert + 
		c_n\left[ 1-(1-2s_n)^n \right]\lvert Q_0\rvert\\
		&\leq&c_n\sum_{Q_j\in \tilde G_1'}\lvert Q_j\rvert + \delta_n\lvert Q_0\rvert\,.
	\end{eqnarray*}
Then $(1-2\delta_n)\lvert Q_0\rvert\leq c_n\sum_{Q_j\in \tilde G_1'}\lvert Q_j\rvert$, and so
	\begin{equation*}
		\sum_{Q_j\in G_1}\lvert Q_j\rvert \leq \frac{c_n}{(1-2\delta_n)}\sum_{Q_j\in 
		\tilde G_1'}\lvert Q_j\rvert.
	\end{equation*}
Hence, using estimate (\ref{assumption}), the construction of generation cubes, the claim at the start of the proof of \ref{thm.a} and Fact \ref{fact.3}, we get
	\begin{equation*}
		\lvert u(P_{Q_0})-u(X)\rvert \geq \frac{\epp}{100},\quad X\in \widetilde Q_j.
	\end{equation*}
	\begin{eqnarray*}
		\int_{ \left\{ (x,\Psi_1(x)): x\in (1-s_n)Q_0 \right\}}\lvert u-u(P_{Q_0})\rvert^2 
		&\leq& C_{\delta,\eta,\lambda,n,(\ref{assumption})}\iint_{\Omega_+\cap T(Q_0)} t\lvert \nabla u\rvert^2 dxdt\\
		&\leq& C_{\delta,\eta,\lambda,n,(\ref{assumption})}\iint_{\leftexp{c}{U_1}\cap T(Q_0)} t\lvert \nabla u\rvert^2 dxdt,
	\end{eqnarray*}
	since $\Omega_+\cap T(Q_0)\subset \leftexp{c}{U_1}\cap T(Q_0)$. Thus,
	\begin{equation*}
		\frac{\epp^2}{100^2}\sum_{Q_j\in \tilde G_1'}\lvert Q_j\rvert \leq 
		C_{\delta,\eta,\lambda,n,(\ref{assumption})}\iint_{T(Q_0)\backslash \cup_{Q\in G_1}T(Q)} 
		t\lvert \nabla u\rvert^2,
	\end{equation*}
	which shows that in case \ref{case.2}, $\frac{\epp^2}{100^2}\sum_{Q_j\in G_1}\lvert Q_j\rvert \leq C_{\delta,n,\lambda,\mu,(\ref{assumption})}\iint_{T(Q_0)\backslash \cup_{Q\in G_1}T(Q)}t\lvert \nabla u\rvert^2$, finishing the proof of Lemma \ref{lma.1}.
\end{proof}
Recall that $Q$ is a ``generation cube'' if $Q\in G_p$ for some $p\geq 1$. We define 
	$G_0=\{Q_0\}$.

	\begin{lemma}{``Packing property''}
		Let $Q$ be a ``dyadic'' cube $\subset Q_0$. Then
		\begin{equation*}
			\sum_{Q_j\subset Q, Q_j\text{ a generation cube}}\lvert Q_j\rvert \leq 
			C_{\lambda,n,\epp,\eta,N,\mu,(\ref{assumption}),(\ref{est.2})}\lvert Q\rvert.
		\end{equation*}
		\label{lma.2}
	\end{lemma}
	\begin{proof}
		Let $M(Q)=\left\{ \text{maximal generation cubes contained in } Q\right\}$, i.e., $Q_1\in M(Q)$ 
		if $Q_1$ is a generation cube and $\nexists Q'$, 
		a generation cube, $Q'\subset Q$ with $Q_1\subsetneq Q'$. Note that the cubes in $M(Q)$ are pairwise disjoint, and any generation cube $Q_j$ contained in $Q$ is contained in a unique maximal $Q_1\in M(Q)$. By disjointness, $\sum_{Q_1\in M(Q)}\lvert Q_1\rvert \leq \lvert Q\rvert$.

		By the construction, we must have 
$$\left\{ Q_j: Q_j\subset Q, Q_j \text{ is a generation cube} \right\}
=\cup_{Q_1\in M(Q)}\cup_{p\geq 0} G_p(Q_1).$$

		Fix $Q$ and fix a maximal generation cube contained in $Q$, $Q_1$. We define 
$G_0:=G_0(Q_1)=\{Q_1\}$, and $G_1:=G_1(Q_1),\,G_2:= G_2(Q_1),...$, etc., 
analogously to $G_p(Q_0)$ above.
We define $U_0=Q_1, U_1=\cup_{Q'\in G_1(Q_1)}Q', U_2=\cup_{Q'\in G_2(Q_1)}Q'$, etc., and 
note that 
$$U_{p+1}=\cup_{Q'\in G_p(Q_1)}U_1(Q')\,.$$
Thus, $\lvert U_{p+1}\rvert =\sum_{Q'\in G_p}\lvert U_1(Q')\rvert$. 
By Lemma \ref{lma.1}, for $p=0,1,\dots$, we have
		\begin{eqnarray*}
			\lvert U_{p+1}\rvert &\leq& C\sum_{Q'\in G_p}\iint_{T(Q')\backslash \cup_{Q''\in G_1(Q')}T(Q'')}t\lvert \nabla u\rvert^2 +(1-\mu)\sum_{Q'\in G_p}\lvert Q'\rvert\\
			&=& C\sum_{Q'\in G_p}\iint_{T(Q')\backslash \cup_{Q''\in G_1(Q')}T(Q'')}t\lvert \nabla u\rvert^2 +(1-\mu)\lvert U_p\rvert\,.
		\end{eqnarray*}
Thus, using the disjointness of the regions $T(Q')\backslash \cup_{Q''\in G_1(Q')}T(Q'')$ for each fixed $p$, in $Q'$ and for consecutive $p$'s, and summing in $p$, we obtain:
		\begin{equation*}
			\sum_{p=0}^\infty \lvert U_{p+1}\rvert \leq C\iint_{T(Q_1)}t\lvert \nabla u\rvert^2+(1-\mu)\sum_{p=0}^\infty\lvert U_p\rvert
		\end{equation*}
		Thus, $\mu\sum_{p=1}^\infty \lvert U_p\rvert \leq C\iint_{T(Q_1)}t\lvert \nabla u\rvert^2+(1-\mu)\lvert Q_1\rvert$ and using (\ref{est.2}), we obtain $\sum_{p=1}^\infty \lvert U_p\rvert \leq C\lvert Q_1\rvert$ or, for each $Q_1\in M(Q)$,
		\begin{equation*}
			\sum_{p\geq 0}\sum_{Q_j\in G_p(Q_1)}\lvert Q_j\rvert \leq 
C_{\lambda,n,\epp,\eta,N,\mu,(\ref{assumption}),(\ref{est.2})}\lvert Q_1\rvert.
		\end{equation*}
		If we now sum over $Q_1\in M(Q)$, Lemma \ref{lma.2} follows.
	\end{proof}

	\subsection{The Stopping Time Construction, Part 2}
	For each generation cube $Q$, we define the corresponding Carleson box $T(Q)$ and the ``rectangle'' $S=S(Q)$. We call the resulting $T(Q)$'s ``generation boxes''. For each generation box 
	$T(Q)$, we define the ``dyadic sawtooth region" 
$\Omega(Q)=T(Q)\backslash \cup_{Q_i\in G_1(Q)}T(Q_i)$. 

Note that if $Q'\subset Q_0$ is a ``dyadic'' sub-cube, then $S=S(Q')$ 
is contained in a unique $\Omega(Q)$.
The uniqueness comes from the fact that if two generation intervals $Q_j$, $Q_i$ are distinct, their associated regions $\Omega(Q_j)$, $\Omega(Q_i)$ have disjoint interiors. The fact that $S$ is contained in some $\Omega(Q)$ 
is due to the fact that if $l_p=\max\left\{ l(Q): Q\in G_p \right\}$, then $l_p\to 0$.

Next, relative to $\reu$, 
for each generation cube $Q$, $\partial \Omega(Q)$ consists of horizontal and vertical ``segments''. The intersection of these ``segments'' with any box $T(Q')$ 
have $H^{n}$ measure adding up to at most $c_n\lvert Q'\rvert$, since 
$H^{n}(\partial T(Q))=c_n\lvert Q\rvert$. Also, the $\left\{ Q_j \right\}$ in $G_1(Q)$, $Q$ a generation cube, are non-overlapping, by maximality. For each generation cube $Q_j$, 
including the unit cube $Q_0$, we define $\varphi_1(z)=u(P_{Q_j})$ on the interior of 
$\Omega(Q_j)$. Thus,
	\begin{equation*}
		\varphi_1(z)=\sum_{p=0}^\infty
		\sum_{Q_j\in G_p} u(P_{Q_j})\chi_{\stackrel{\circ}{\Omega}(Q_j)}.	
	\end{equation*}
We consider now $\lvert \nabla \varphi_1(z)\rvert$. As a distribution on $\RR_+^{n+1}$, 
$$\nabla \varphi_1=\sum_{p=0}^\infty\sum_{Q_j\in G_p}u(P_{Q_j})
\nabla \chi_{\stackrel{\circ}{\Omega} (Q_j)}\,.$$
It is easy to see that $\lvert \nabla \chi_{\stackrel{\circ}{\Omega}(Q_j)}\rvert = dH^{n}\lfloor_{\,\Sigma_j}$, 
where $\Sigma_j=\left\{ t>0 \right\}\cap \partial \Omega(Q_j)$. 
Since $\lvert u(P_{Q_j})\rvert\leq 1, \lvert \nabla \varphi_1\rvert\leq \sum_{p=0}^\infty
\sum_{Q_j\in G_p}\lvert \nabla \chi_{\stackrel{\circ}{\Omega}(Q_j)}\vert$. Thus, for fixed $Q$, we have
	\begin{equation*}
		\iint_{T(Q)}\lvert \nabla \varphi_1\rvert \leq \sum_{p,j} H^{n}(T(Q)\cap \Sigma_j).
	\end{equation*}

\begin{claim}
	\begin{equation*}
		\sum_{p,j}H^{n}(T(Q)\cap \Sigma_j)\leq 
		C_{\epp,\lambda,\mu,(\ref{assumption}),(\ref{est.2})}\lvert Q\rvert.
	\end{equation*}
\end{claim}

	To see this, first consider those $Q_j$ such that 
$$T(Q)\cap \Sigma_j=T(Q)\cap \partial \Omega(Q_j)\cap \left\{ t>0 \right\}\neq \emptyset,$$ 
but such that $Q_j\nsubseteq Q$. In this case, assume first that $l(Q_j)\leq l(Q)$. Then, 
$T(Q)\cap \Sigma_j$ is a union of ``intervals'' along a ``vertical" side of $T(Q)$. 
These ``intervals" are pairwise disjoint, so they contribute at most $c_nH^{n}(\partial T(Q))$. 
If $l(Q_j)>l(Q)$, there are at most $c_n$ such cubes, each contributes at most 
$c_n H^{n}(\partial T(Q))$. Next we consider generation cubes such that 
$Q_j\subset Q$. Then, 
$$\sum_{Q_j\subset Q}H^{n}(T(Q)\cap \Sigma_j)\leq c_n\sum_{Q_j\subset Q}\lvert Q_j\rvert\leq 
C_{\epp,\lambda,n,(\ref{assumption}),(\ref{est.2})}\lvert Q\rvert\,,$$ by 
Lemma \ref{lma.2}. Thus, $\lvert\nabla \varphi_1\rvert$ is a Carleson measure.

We now say that $S=S(Q)$ is a blue ``rectangle'' if
	\begin{equation*}
		\sup_{X,Y\in S}\lvert u(X)-u(Y)\rvert \leq \frac{\epp}{10}.
	\end{equation*}
	Otherwise, we say that $S$ is a red ``rectangle''. Assume that $S=S(Q)$ is a blue ``rectangle''. Let $Q_j$ be the unique generation cube so that $S(Q)\subset \Omega(Q_j)$. 
Because $S(Q)\subset \Omega(Q_j)$, $\lvert u(P_Q)-u(P_{Q_j})\rvert <\frac{\epp}{10}$. 
Since $P_Q\in S(Q)$, if $X\in S(Q)$, then
$\lvert u(P_Q)-u(X)\rvert <\frac{\epp}{10}$. Hence, $\lvert u(X)-u(P_{Q_j})\rvert \leq\frac{\epp}{5}$ 
for $X\in S(Q)$. But, $\varphi_1(X)=u(P_{Q_j})$ on $\Omega(Q_j)$, so that $\lvert \varphi_1(X)-u(X)\rvert \leq \frac{\epp}{5}$ on every blue $S$.

	The final step is to correct $\varphi_1$ in the red rectangles. $S=S(Q)$ is red if there exists $X_0, Y_0\in S$ such that
	\begin{equation*}
		\lvert u(X_0)-u(Y_0)\rvert > \frac{\epp}{10}.
	\end{equation*}
	Let $\widetilde S$ be the slightly fattened version of $S$,
as at the start of this section. By (\ref{eqn.4}),
	\begin{eqnarray*}
		\frac{\epp^2}{100}&\leq& C_{\lambda,n}^2\left( \frac{\lvert X_0-Y_0\rvert}{l(Q)} \right)^{2\alpha}l(Q)^2\aaviint_{\!\!\widetilde{S}}\,\lvert \nabla u\rvert^2\\
		&\leq& C_{\lambda,n}^2\frac{1}{l(Q)^n}\iint_{\widetilde S}t\,|\nabla u\rvert^2
	\end{eqnarray*}
	or
	\begin{equation*}
		|Q|\leq \frac{C_{\lambda,n}^2}{\epsilon^2}\iint_{\widetilde S}t\,\lvert \nabla u\rvert^2.
	\end{equation*}
	By the bounded overlap of $\left\{ \widetilde S \right\}$, we have:
	\begin{equation*}
		\sum_{Q_k\subset Q:  \,S(Q_k)\text{ red}}|Q_k|\leq \frac{C_{\lambda,n,(\ref{est.2})}^2}{\epsilon^2}\,\lvert Q\rvert\,, 
	\end{equation*}
	in view of estimate (\ref{est.2}).
Also, if $S$ is red, then by (\ref{eqn.3}) (with $S=S(Q)$),
	\begin{eqnarray*}
		\iint_S\lvert \nabla u\rvert &\leq& 
\left( \iint_S\lvert \nabla u\rvert^2 \right)^{\frac{1}{2}}l(Q)^{\frac{n+1}{2}}\\
		&\leq&\frac{C_{\lambda,n}}{l(Q)}\left( \iint_{\widetilde S}\lvert u\rvert^2 \right)^{\frac{1}{2}}l(Q)^{\frac{n+1}{2}}\\
		(\text{since } \lVert u\rVert_\infty\leq 1)&\leq& \frac{C_{\lambda,n}}{l(Q)}\cdot l(Q)^{n+1}
		\leq \frac{C_{\lambda,n}}{\epsilon^2}\iint_{\widetilde S}t\lvert \nabla u\rvert^2,
	\end{eqnarray*}
	by the previous estimate.

	Then, if $\RRR=\cup_{S=S(Q'), \,S\text{ red}}\,S$, and we consider $\lvert \nabla u\rvert\, 
\chi_{\RRR}$, 
also note that $T(Q)\cap \RRR=\cup_{Q':S(Q')\subset T(Q), \,S(Q')\text{ red}}\,S(Q')$. Then,
	\begin{eqnarray*}
		\iint_{T(Q)}\lvert \nabla u\rvert\, \chi_{\RRR} &=& \sum_{S=S(Q')\subset T(Q), \,S\text{ red}}\iint_{S(Q')}\lvert \nabla u\rvert\\
		&\leq& \sum\frac{C_{\lambda,n}}{\epsilon^2}\iint_{\widetilde S(Q')}t\lvert \nabla u\rvert^2\leq \frac{C_{\lambda,n}}{\epsilon^2}\iint_{T(\widetilde Q)}t\lvert \nabla u\rvert^2\\
		&\leq& C_{\lambda,n,(\ref{est.2})}\,|Q|
	\end{eqnarray*}
	by (\ref{est.2}), so that $\lvert \nabla u\rvert\, \chi_{\RRR}$ is a Carleson measure.

	Define now
	\begin{equation*}
		\varphi_2(z)=
		\begin{cases}
			\varphi_1(z), &z\notin \RRR\\
			u(z), &z\in\RRR
		\end{cases}
	\end{equation*}
We clearly have $\lvert u(z)-\varphi_2(z)\rvert \leq \epsilon$. Also, 
$\nabla \varphi_2(z)=\chi_{\RRR}\nabla u+\chi_{(T(Q_0)\backslash \RRR)}\nabla \varphi_1+J$, where $J$ accounts for the jumps of $\varphi_2$ as $z$ 
crosses $\partial \RRR\cap\reu$. Since $\lvert \varphi_2\rvert \leq 1+\epsilon, \,J$ is a measure dominated by $(1+\epsilon)\,dH^{n}\lfloor_{\,\partial \RRR}$. This last measure is Carleson by a previous estimate.   
This proves Theorem \ref{thm.a}.
\end{proof}

\end{document}